\documentclass{amsart}
\usepackage{marktext}
\usepackage{gimac}

\usepackage[tight]{subfigure}

\DeclareMathOperator{\vol}{vol}

\colorlet{darkred}{red!40!black}

\newlength{\picwidth}

\newcommand{\loc}{\mathrm{loc}}

\newdelim{\ip}{\langle}{\rangle}

\makeatletter
\newcommand{\leqnomode}{\tagsleft@true\let\veqno\@@leqno}
\newcommand{\reqnomode}{\tagsleft@false\let\veqno\@@eqno}
\newcommand{\mylabel}[2]{\def\@currentlabel{#2}\label{#1}}
\makeatother

\newcommand{\sfrac}[2]{#1/#2}
\newcommand{\FE}{\mathcal{E}}  

\begin{document}
\title[Phase separation in the advective Cahn--Hilliard equation]%
  {Phase separation in the \\ advective Cahn--Hilliard equation}
\author[Feng]{Yu Feng}
\address{%
  Department of Mathematics, University of Wisconsin --
  Madison, Madison, WI 53706}
\email{feng65@wisc.edu}

\author[Feng]{Yuanyuan Feng}
\address{%
  Department of Mathematical Sciences,
  Carnegie Mellon University,
  Pittsburgh, PA 15213}
\email{yuanyuaf@math.cmu.edu}

\author[Iyer]{Gautam Iyer}
\address{%
  Department of Mathematical Sciences,
  Carnegie Mellon University,
  Pittsburgh, PA 15213}
\email{gautam@math.cmu.edu}

\author[Thiffeault]{Jean-Luc Thiffeault}
\address{%
  Department of Mathematics, University of Wisconsin --
  Madison, Madison, WI 53706}
\email{jeanluc@math.wisc.edu}
\thanks{%
  This material is based upon work partially supported by
  the National Science Foundation, under grant
  DMS-1814147,
  and the Center for Nonlinear Analysis.
}
\begin{abstract}
  The Cahn--Hilliard equation is a classic model of phase separation in binary mixtures that exhibits spontaneous coarsening of the phases.
  We study the Cahn--Hilliard equation with an imposed advection term in order to model the stirring and eventual mixing of the phases.
  The main result is that if the imposed advection is sufficiently mixing then no phase separation occurs, and the solution instead converges exponentially to a homogeneous mixed state.
  The mixing effectiveness of the imposed drift is quantified in terms of the dissipation time of the associated advection-hyperdiffusion equation, and we produce examples of velocity fields with a small dissipation time.
  We also study the relationship between this quantity and the dissipation time of the standard advection-diffusion equation.
\end{abstract}
\maketitle

\section{Introduction}

Spinodal decomposition refers to the phase separation of a binary mixture, such as alloys that are quenched below their critical temperature.  A well-studied model is the Cahn--Hilliard equation~\cites{CahnHilliard58,Cahn61}, where the evolution of the normalized concentration difference~$c$ between the two phases is governed by the equation
\begin{equation}\label{e:ch}
 \partial_t c +\gamma D\,\lap^2 c = D \lap(c^3 - c)\,.
\end{equation}
Here $D > 0$ is a mobility parameter, and $\sqrt{\gamma}$ is the Cahn number, which is related to the surface tension at the interface between phases.  The coefficient $\gamma D$ is a hyperdiffusion that regularizes the equation at small length scales by overcoming the destabilizing~$-D\lap c$ term.  The concentration $c$ is normalized such that the regions $\set{c = 1}$ and $\set{c = -1}$ represent domains that are pure in each phase.  For simplicity, we will only consider~\eqref{e:ch} on the $d$-dimensional torus~$\T^d$.

\setlength{\picwidth}{.23\linewidth}
\begin{figure}[ht]
  \subfigure{\includegraphics[width=\picwidth]{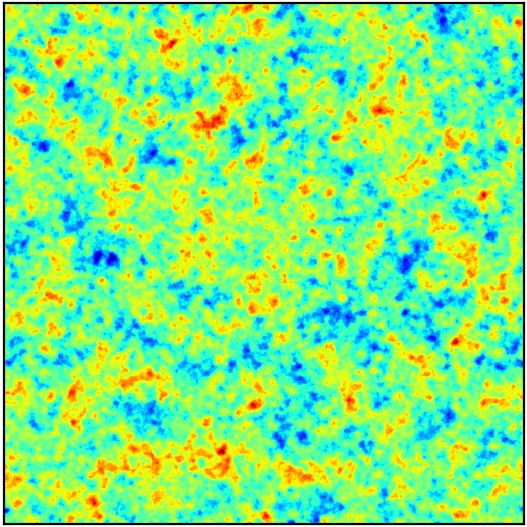}}
  \subfigure{\includegraphics[width=\picwidth]{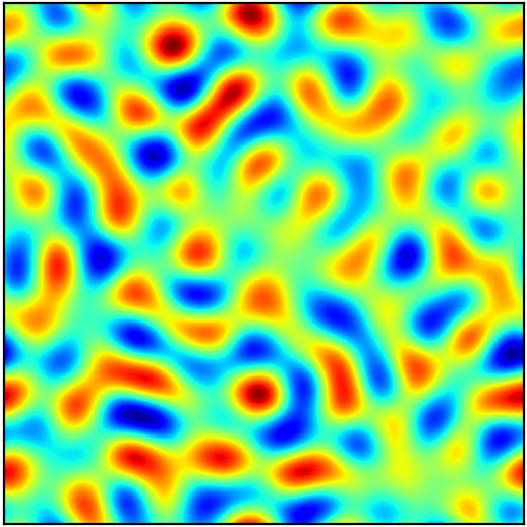}}
  \subfigure{\includegraphics[width=\picwidth]{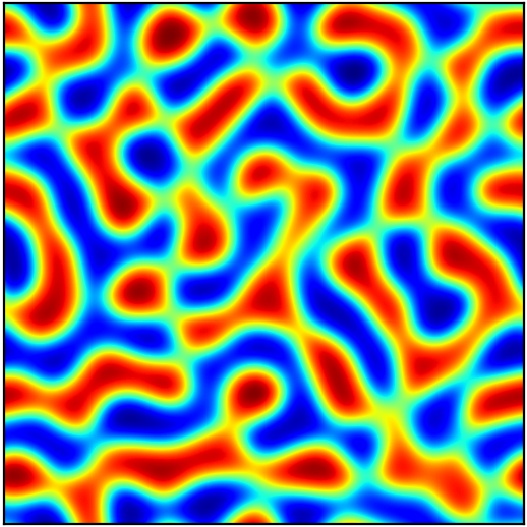}}
  \subfigure{\includegraphics[width=\picwidth]{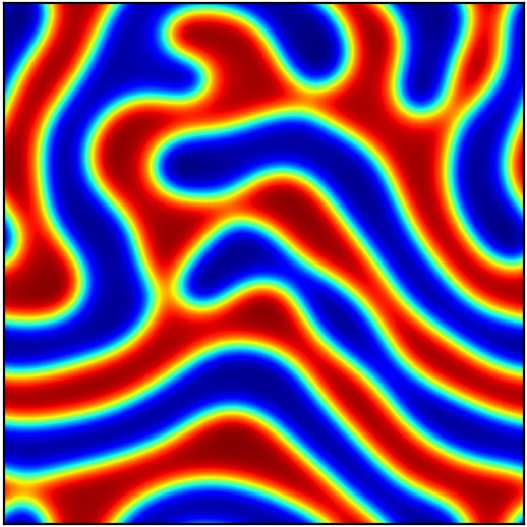}}
  \caption{The solution of the Cahn--Hilliard equation~\eqref{e:ch} on the 2D torus with $D = 0.001$, $\gamma = .01 / (2\pi)^2$ at times~$t=0$, $2$, $5$ and $20$.  The phases $c=1$ and $c=-1$ are red and blue, respectively, and $c=0$ is green.}
  \label{f:CH}
\end{figure}

When~$\gamma$ is small, solutions to~\eqref{e:ch} spontaneously form domains with $c = \pm 1$ separated by thin transition regions (see Figure~\ref{f:CH}).
This has been well studied by many authors (see for instance~\cites{ElliottSongmu86,Elliott89,Pego89}), and the underlying mechanism can be understood as follows.  The free energy of this system, $\FE$, can be decomposed into the sum of the chemical free energy, $\FE_\mathrm{chem}$, and the interfacial free energy, $\FE_\mathrm{int}$, where
\begin{equation*}
  \FE_\mathrm{chem}
    \defeq \tfrac{1}{4} \int_{\T^d} (c^2 - 1)^2 \, dx
  \qquad\text{and}\qquad
  \FE_\mathrm{int} \defeq
    \tfrac{1}{2}\gamma \int_{\T^d} \abs{\grad c}^2  \, dx\,.
\end{equation*}
Using~\eqref{e:ch}, one can directly check that $\FE$ decreases with time, and hence solutions should approach minimizers of~$\FE$ after a long time.
Minimizing the chemical free energy $\FE_\mathrm{chem}$ favors forming domains where $c = \pm 1$.
Minimizing the interfacial free energy~$\FE_\mathrm{int}$ favors interfaces of thickness~$\sqrt{\gamma}$ separating the domains.
As a result, the typical behavior of equation~\eqref{e:ch} is to spontaneously phase-separate as in Figure~\ref{f:CH}.

In this paper we study the effect of stirring on spontaneous phase separation.
When subjected to an incompressible stirring velocity field~$u(t,x)$, equation~\eqref{e:ch} is modified to
\begin{equation}\label{e:ach}
  \partial_t c + u \cdot\grad c + \gamma\lap^2 c
  = \lap (c^{3}-c) \,.
\end{equation}
For simplicity, we have set the mobility parameter~$D$ to be~$1$.
The advective Cahn--Hilliard equation~\eqref{e:ach} has has been studied by many authors~\cites{ChanPerrotEA88,LaugerLaubnerEA95,ONaraighThiffeault07,ONaraighThiffeault07a,ONaraighThiffeault08,LiuDedeEA13} for both passive and active advection.
Under a strong shear flow, for instance, it is known that solutions to~\eqref{e:ach} equilibrate along the flow direction and spontaneously phase separates in the direction perpendicular to the flow~\cites{Berthier01, Bray03,ShouChakrabarti00,HashimotoMatsuzakaEA95}.
\smallskip

Our main result is to show that if the stirring velocity field is sufficiently mixing, then no phase separation occurs.
More precisely, we show that if the \emph{dissipation time} of~$u$ is small enough, then~$c$ converges exponentially to the total concentration~$\bar c = \int_{\T^d} c_0 \, dx$, where $c_0$ denotes the initial data.
This is illustrated by the numerical simulations in Figure~\ref{f:chmixed}, where the velocity field~$u$ was chosen to be alternating horizontal and vertical shear flows with randomized phases (see~\cite{Pierrehumbert94,ONaraighThiffeault07,ONaraighThiffeault08}).
When the shear amplitude, $A$, is small, the norms of the solution settle to some non-zero value after a large time.
As the amplitude is increased, the flow mixes faster, and we see the solution decays exponentially to $\bar c = 0$.

\begin{figure}[htb]
  \includegraphics[width=.45\linewidth]{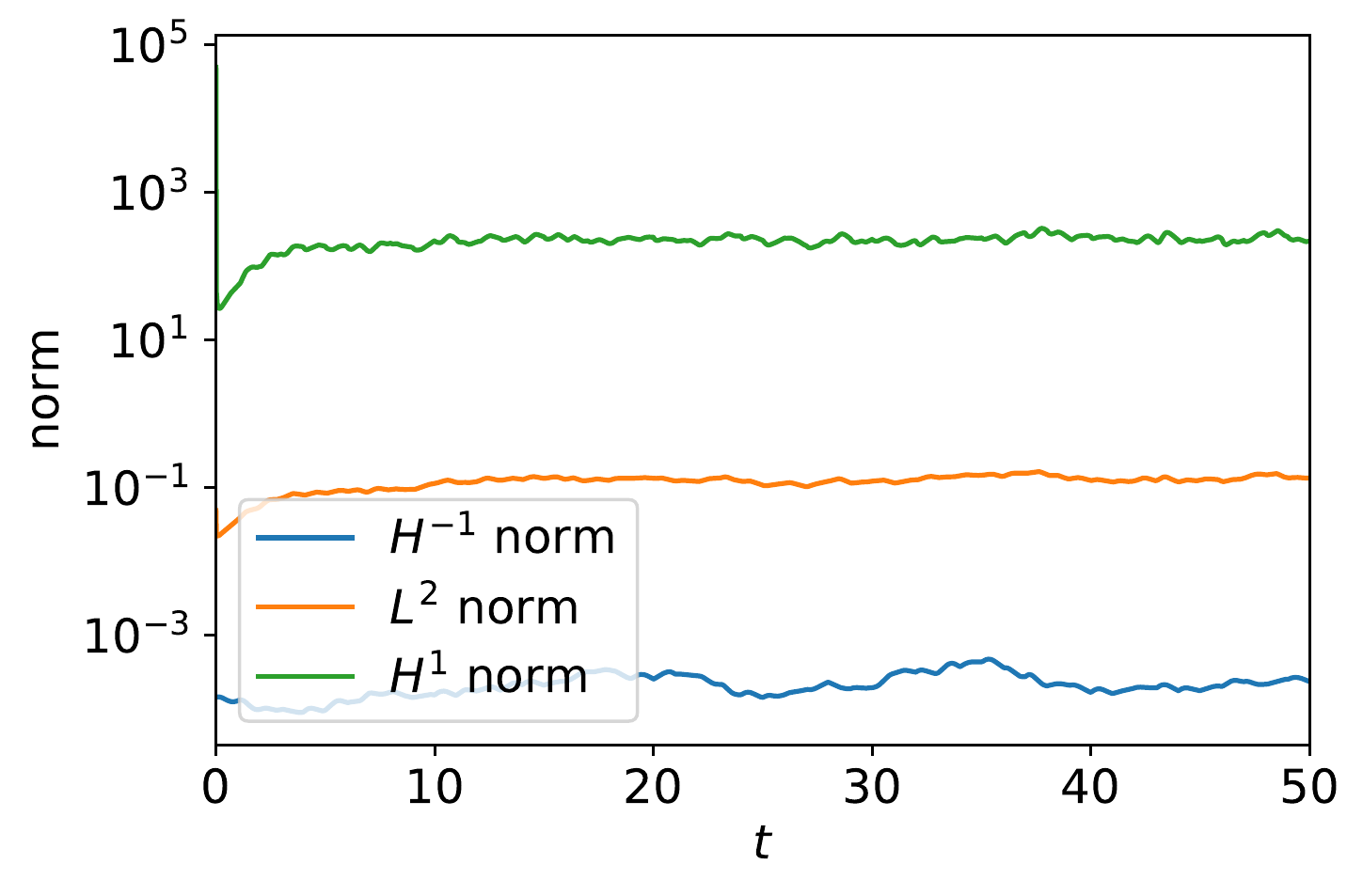}
  \hspace{.025\linewidth}
  \includegraphics[width=.45\linewidth]{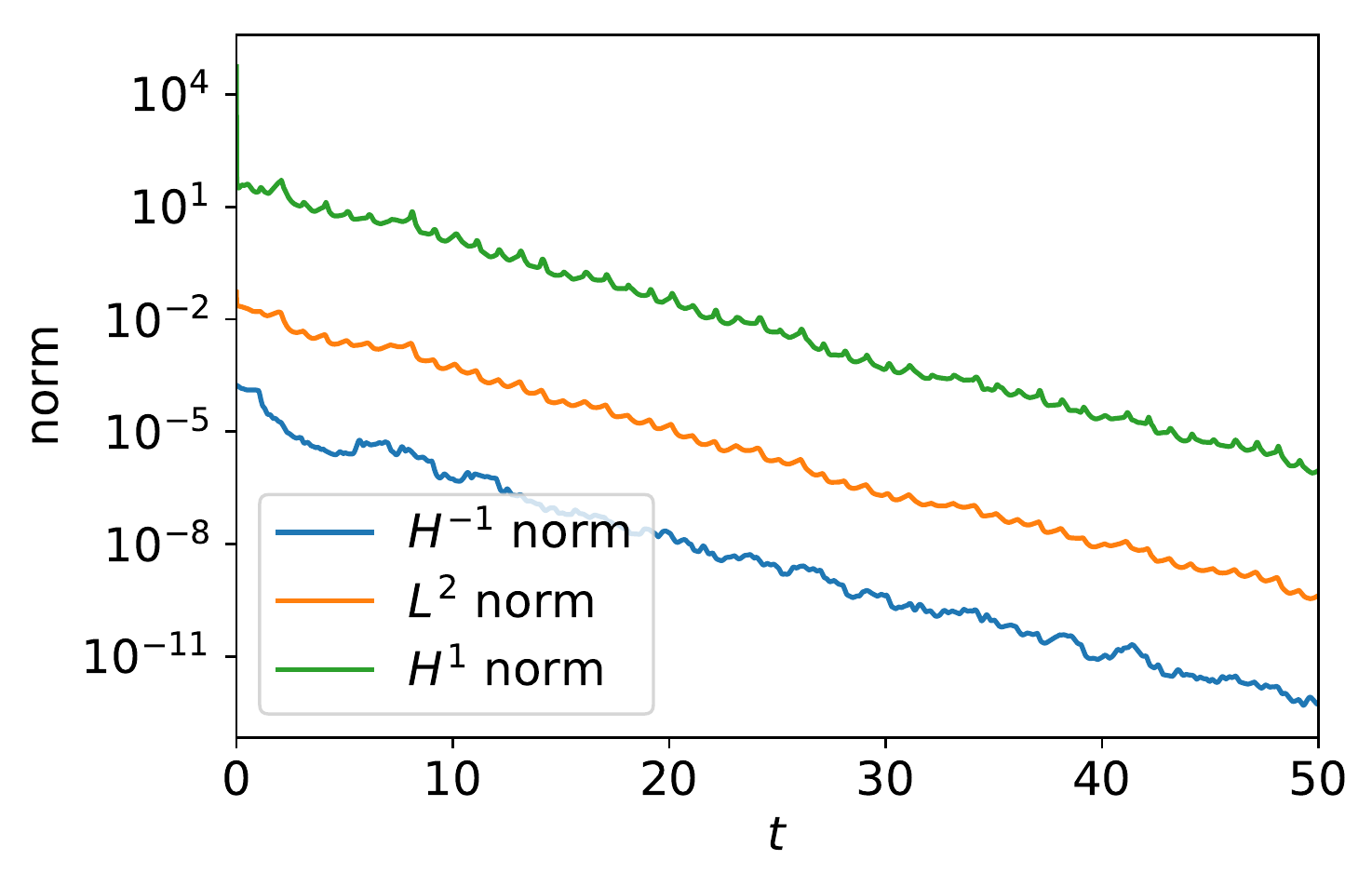}
  \caption{Decay of $H^1$, $L^2$, and $H^{-1}$ norms for the random shear flow for $A=0.5$ (left) and $A=2$ (right).  On the left the norms settle to equilibrium values; on the right they decay exponentially.}
  \label{f:chmixed}
\end{figure}
\begin{figure}[htb]
  \includegraphics[height=.3\linewidth]{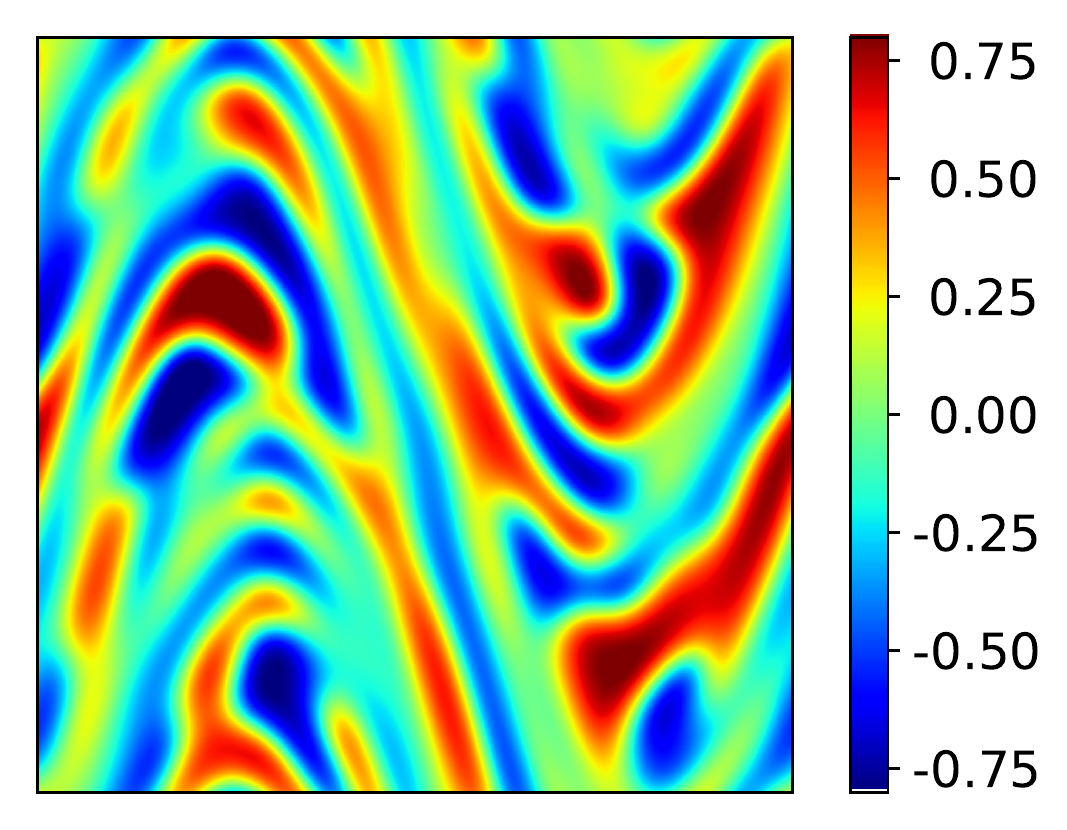}
  \quad
  \includegraphics[height=.3\linewidth]{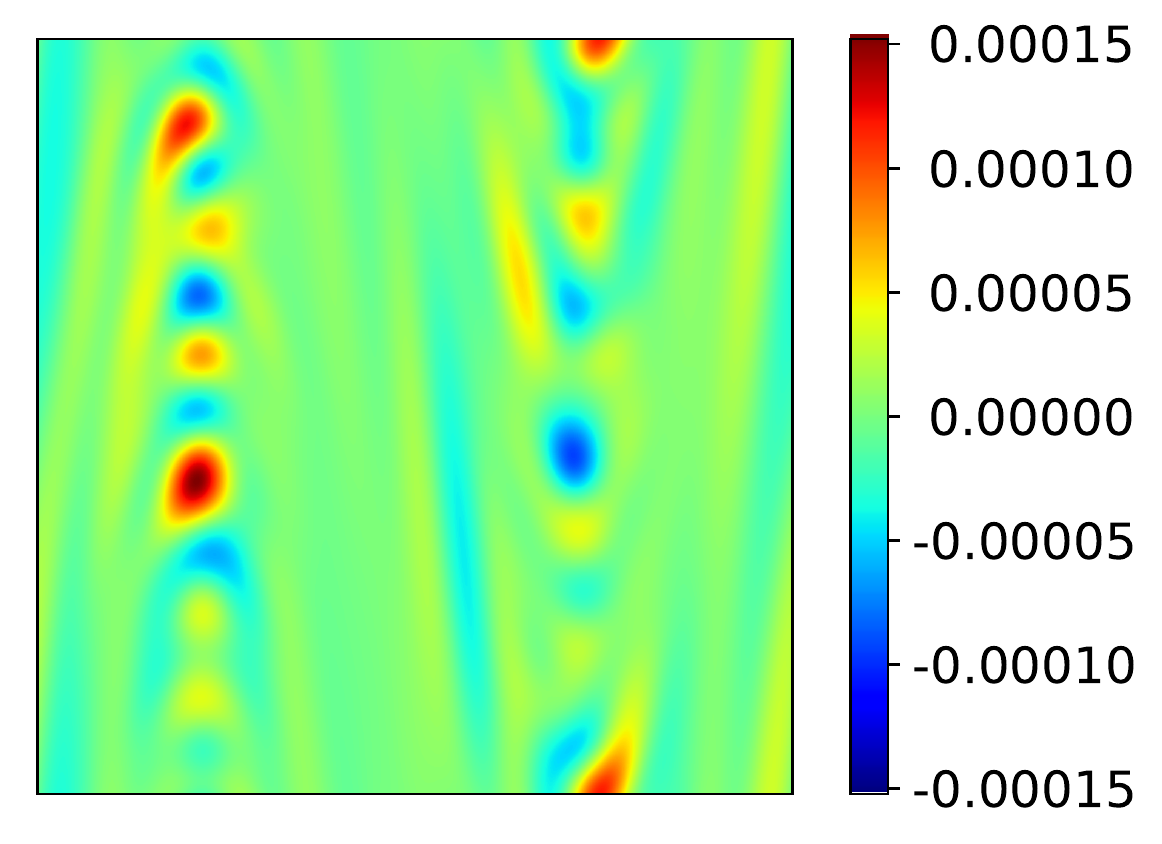}
  \caption{%
    Final concentration $c$ for the two cases in Figure~\ref{f:chmixed}.
  }
  \label{f:chmixed_final}
\end{figure}

\subsection{Decay of the advective Cahn--Hilliard equation}\label{s:introCHdecay}

To state our main result, we need to first introduce the notion of \emph{dissipation time}.  Let $u$ be a divergence-free vector field and consider the equation
\begin{equation}\label{e:ad}
  \partial_t \theta + u \cdot \grad \theta + \gamma(-\lap)^{\alpha} \theta = 0,
\end{equation}
with $\alpha > 0$, periodic boundary conditions, and mean-zero initial data.
For~$\alpha=1$ this is the \emph{advection-diffusion equation}; for $\alpha=2$ it is the \emph{advection-hyperdiffusion equation}.  Incompressibility of~$u$ and the Poincar\'e inequality immediately imply that $\norm{\theta_t}_{L^2}$ is decreasing as a function of~$t$, and
\begin{equation}\label{e:ee}
  \norm{\theta(s+t)}_{L^2} \leq e^{-(2\pi)^{2\alpha} \gamma t}\, \norm{\theta(s)}_{L^2}\,.
\end{equation}
Thus, we are guaranteed
\begin{equation}\label{e:dtime1}
\norm{\theta(s + t)}_{L^2} \leq \tfrac{1}{2} \norm{\theta(s)}_{L^2}\,,
\quad\text{for every}\quad
t \geq \frac{\ln 2}{ (2\pi)^{2\alpha} \gamma}\,,
\end{equation}
and every $s \geq 0$.
However, $u$ generates gradients through filamentation, which causes solutions to dissipate $\norm{\theta(t)}_{L^2}$ faster.  This may result in the lower bound in~\eqref{e:dtime1} being attained at much smaller times, and the smallest time~$t$ at which this happens is known as the \emph{dissipation time} (see for instance~\cites{FannjiangWoowski03,FengIyer19}).

\begin{definition}[Dissipation time]\label{d:dissipationTime}
  Let $\mathcal S_{s, t}^{u, \alpha}$ be the solution operator to~\eqref{e:ad} on $\T^d \times (0, \infty)$.
  That is, for any $f \in L^2(\T^d)$, the function $\theta(t) = \mathcal S_{s, t}^{u, \alpha} f$ solves~\eqref{e:ad} with initial data $\theta(s) = f$, and periodic boundary conditions.
  The \emph{dissipation time} of~$u$ is
  \begin{equation}\label{e:dtimeDef}
    \tau^*_\alpha(u, \gamma) \defeq  \inf \set[\Big]{ t \geq 0 \st  \norm{ \mathcal S^{u, \alpha}_{s, s+t} }_{\dot L^2 \to \dot L^2 } \leq \tfrac{1}{2} \text{ for all $s\geq 0$} }\,.
  \end{equation}
  Here $\dot L^2$ is the space of all mean-zero, square integrable functions on the torus~$\T^d$.
\end{definition}

While this definition makes sense for any~$\alpha > 0$, we are mainly interested in the case when~$\alpha$ is either~$1$ or~$2$.
Note that~\eqref{e:dtime1} implies $\tau^*_\alpha( u, \gamma ) \leq O(1/\gamma)$ as $\gamma \to 0$.
If, however, $u$ is mixing, then this can be dramatically improved (see for instance~\cites{ConstantinKiselevEA08,Zlatos10,CotiZelatiDelgadinoEA18,Wei18,FengIyer19,Feng19}).
In fact \cite{FengIyer19} bound $\tau^*_1(u, \gamma)$ explicitly in terms of the mixing rate of~$u$.
Moreover, when~$u$ is exponentially mixing, \cites{CotiZelatiDelgadinoEA18,Wei18,Feng19} show that $\tau^*_1(u, \gamma) \leq O( \abs{\ln \gamma}^2 )$ as $\gamma \to 0$.
\medskip

With this notion, we can now state our main result.
\begin{theorem}\label{t:chmix}
Let $d \in \set{2, 3}$, $u\in L^\infty ((0,\infty);W^{1,\infty}(\T^d))$, and $c$ be the solution of~\eqref{e:ach} with initial data~$c_0\in H^2(\T^d)$.
\begin{asparaenum}
  \item
    When~$d = 2$, for any  $\beta > 1$, $\mu>0$, there exists a time
    \begin{equation*}
      T_0
	= T_0( \norm{c_0 - \bar c}_{L^2}, \bar c, \beta, \gamma, \mu)
    \end{equation*}
    such that if $\tau^*_2(u, \gamma) < T_0$, then for every~$t \geq 0$, we have
    \begin{align}\label{e:globall2decay}
      \norm{c(t)-\bar c}_{L^2}\leq \beta e^{-\mu t}\norm{c_0-\bar c}_{L^2}\,.
    \end{align}

  \item
    When~$d = 3$, for any $\beta > 1$, $\mu>0$, there exists a time
    \begin{equation*}
      T_1
	= T_1( \norm{c_0 - \bar c}_{L^2}, \bar c, \beta, \gamma, \mu)
    \end{equation*}
    such that if
    \begin{equation}\label{e:3dSmallness}
      (1 + \norm{\grad u}_{L^\infty})^{1/2} \tau^*_2(u, \gamma) < T_1\,,
    \end{equation}
    then~\eqref{e:globall2decay} still holds for every $t \geq 0$.
\end{asparaenum}
\end{theorem}
\begin{remark}
  The times $T_0$ and $T_1$ can be computed explicitly, as can be seen from the proof of the theorem and equations~\eqref{e:T0prime} and \eqref{e:T1prime}, below.

We emphasize that $T_0$ and $T_1$ only depend on the mean concentration $\bar c$, the variance of the initial data $\norm{c_0 - \bar c}_{L^2}^2$, the Cahn number $\sqrt{\gamma}$ and the exponential decay constants $\beta, \mu$.
    Once $T_0$ and $T_1$ are determined from these parameters, in order to apply Theorem~\ref{t:chmix}, we need to produce velocity fields $u$ that satisfy~$\tau^*_2(u, \gamma) < T_0$ when $d = 2$, and the condition~\eqref{e:3dSmallness} when~$d = 3$.
    We do this in Section~\ref{s:introSmallDTime}, below by using sufficiently mixing flows with a large amplitude.
    In general, however, smallness of the dissipation time (such as the conditions required in Theorem~\ref{t:chmix}) are weaker than mixing, and there may be simpler examples of velocity fields that satisfy the requirements of Theorem~\ref{t:chmix}.

\end{remark}

Several authors have used mixing properties of the advection term to quench reactions, prevent blow-up, and stem the growth of non-linear PDEs (see for instance~\cites{FannjiangKiselevEA06,HouLei09,BerestyckiKiselevEA10,KiselevXu16,BedrossianHe17,IyerXuEA19}).
Our results are similar in spirit to those in~\cite{IyerXuEA19}, where the authors used related ideas to prove decay of solutions to a large class of nonlinear parabolic equations.
These results were formulated for second-order PDEs where the diffusive term is the Laplacian, but they can easily be generalized to apply when the diffusive term is the bi-Laplacian as we have in~\eqref{e:ach}.
Unfortunately, the assumptions required for the results in~\cite{IyerXuEA19} to apply are not satisfied by the nonlinear term, even when~$d = 2$, and thus we cannot use them here.

Our 3D result is qualitatively different (and weaker) from the 2D case, and from the results in~\cites{IyerXuEA19}.
Indeed, Theorem~\ref{t:chmix} in 2D and all the results in~\cites{IyerXuEA19} only rely on smallness of the dissipation times $\tau^*_1$ or $\tau^*_2$.
In 3D, however, Theorem~\ref{t:chmix} now requires smallness of $(1 + \norm{\grad u})_{L^\infty}^{1/2} \tau^*_2$.
  The reason for this is that in 2D we are able to estimate the nonlinear term $\norm{\lap( c^3 - c) }_{L^2} $ by $\norm{c - \bar c}_{L^2} \norm{\lap c}_{L^2}^2$ in 2D, and by $\norm{\grad c}_{L^2} \norm{\lap c}_{L^2}^2$ in 3D.
  The growth of $\norm{c - \bar c}_{L^2}$ can easily be controlled independent of the advecting flow, and so the 2D result can be formulated only in terms of the dissipation time $\tau^*_2(u, \gamma)$.
  The quantity $\norm{\grad c}_{L^2}$, however, is expected to depend intrinsically on (and grow with) the advecting flow, and the result in 3D involves both $\tau^*_2(u, \gamma)$ and the size of the flow (condition~\eqref{e:3dSmallness}).

In the next section we produce velocity fields where this is arbitrarily small.
We remark, however, that while we can find velocity fields for which $(1 + \norm{\grad u}_{L^\infty})^{1/2} \tau^*_2$ is arbitrarily small, it appears impossible to produce velocity fields for which $(1 + \norm{\grad u}_{L^\infty}) \tau^*_2$ is arbitrarily small.
To see this, the proof in~\cite{Poon96} (see also equation~(9) in~\cite{MilesDoering18}) can be easily adapted to obtain the lower bound
\begin{equation*}
  \tau^*_2(u, \gamma) \geq
    \frac{1}{C \norm{u}_{C^2}}
    \ln \paren[\Big]{ 1 +
      \frac{C \norm{u}_{C^2}}{\gamma}
    }
\end{equation*}
for some explicit dimensional constant~$C$.
(Here by $\norm{u}_{C^2}$ we mean the spatial $C^2$ norm $\sup_t \norm{u}_{C^2(\T^d)}$.)  When $\tau^*_2(u, \gamma)$ is small, we expect $\norm{u}_{C^2}$ to be large, and in this case the above shows $(1 + \norm{u}_{C^2}) \tau^*_2(u, \gamma)$ grows at least logarithmically with~$\norm{u}_{C^2}$.

\subsection{Incompressible velocity fields with small dissipation time}\label{s:introSmallDTime}

In order to apply Theorem~\ref{t:chmix}, we need to produce incompressible velocity fields~$u$ for which $\tau^*_2(u, \gamma)$ is arbitrarily small when $d = 2$, and for which $(1 + \norm{\grad u}_{L^\infty}^{1/2}) \tau^*_2(u, \gamma)$ is arbitrarily small when~$d = 3$.
We do this here by rescaling \emph{mixing flows}.  This has been studied previously by~\cites{ConstantinKiselevEA08,KiselevShterenbergEA08,Zlatos10,CotiZelatiDelgadinoEA18,FengIyer19,Feng19} when the diffusive term is the standard Laplacian.
With minor modification, the proofs can be adapted to our context, where the diffusive term is the bi-Laplacian.

\begin{proposition}\label{p:dtimeMix}
  Let~$v \in L^\infty( [0, \infty); C^2( \T^d ) )$, and define $u_A(x, t) = A v(x, At)$.
  If $v$ is \emph{weakly mixing} with rate function~$h$, then
  \begin{equation*}
    \tau^*_2( u_A, \gamma ) \xrightarrow{A \to \infty} 0 \, .
  \end{equation*}
  If further~$v$ is \emph{strongly mixing} with rate function $h$, and
  \begin{equation}\label{e:thtVanish}
    t\, h(t) \xrightarrow{t \to \infty} 0\,,
  \end{equation}
  then
  \begin{equation*}
    (1 + \norm{\grad u_A}_{L^\infty})^{1/2} \tau^*_2( u_A, \gamma ) \xrightarrow{A \to \infty} 0\,.
  \end{equation*}
\end{proposition}

For ease of presentation, we defer the definition of weak and strong mixing used above to Section~\ref{s:dtimeMix} (see Definition~\ref{d:mixing}, below).
To the best of our knowledge the existence of smooth, \emph{time-independent} or even \emph{time-periodic},  mixing flows on the torus is open.
Various interesting and explicit examples of mixing flows were constructed in \cites{YaoZlatos17,AlbertiCrippaEA19,ElgindiZlatos19}.
Unfortunately none of these examples are spatially regular enough to be used in Proposition~\ref{p:dtimeMix}.

Fortunately, there are many known examples of (spatially) smooth, \emph{time-dependent}, flows on the torus that are exponentially mixing, and any such flow \emph{will} satisfy the conditions required by Proposition~\ref{p:dtimeMix}.
The simplest example we are aware of is to use alternating horizontal/vertical sinusoidal shear flows with randomized phases.
These were introduced by Pierrehumbert~\cite{Pierrehumbert94} and used to produce our Figure~\ref{f:chmixed}.
One can show that these flows, and a variety of other examples, are exponentially mixing using techniques in~\cite{BedrossianBlumenthalEA19}.

We also remark that the mixing requirement in Proposition~\ref{p:dtimeMix} is morally much stronger than what is needed in order to apply Theorem~\ref{t:chmix}.
Indeed, for Theorem~\ref{t:chmix} one only needs flows whose dissipation time $\tau^*_2$ is sufficiently small.
Proposition~\ref{p:dtimeMix} ensures smallness of~$\tau^*_2(u_A, \gamma)$ by using the property that the flow $v$ sends a fraction of the total energy to high frequencies, which then gets rapidly damped by the diffusion.
The mixing assumptions on~$v$, however, ensure a much stronger property, namely that the flow~$v$ eventually sends all the energy to high frequencies (see~\cite{DrivasElgindiEA19} for a longer discussion).
Thus the mixing hypothesis in Proposition~\ref{p:dtimeMix} is most likely much stronger than what may be needed to apply Theorem~\ref{t:chmix}.
In theory, it should also be easier to find flows directly satisfying the requirements of Theorem~\ref{t:chmix}, without using Proposition~\ref{p:dtimeMix}.

When the diffusion operator is the standard Laplacian) this was done in~\cite{IyerXuEA19}.
Here, the authors showed that for any $\tau_0> 0$, there exists a sufficiently strong and fine cellular flow, $u$, for which~$\tau^*_1(u, \gamma) < \tau_0$.
This provides a simple, explicit, smooth, \emph{time independent} family of velocity fields with arbitrarily small dissipation time (when the diffusion operator is the standard Laplacian), and in~\cite{IyerXuEA19} the authors used it to prevent blow up in the Keller--Segel and other second-order, non-linear, parabolic PDEs.

We expect that for any $\tau_0 > 0$, one can also construct sufficiently strong and fine cellular flows for which~$\tau^*_2(u, \gamma) < \tau_0$ (we recall here $\tau^*_2(u, \gamma)$ is the dissipation time when the diffusion operator is the bi-Laplacian).
Unfortunately the proof in~\cite{IyerXuEA19} does not generalize, and thus we are presently unable to produce cellular flows for which~$\tau^*_2(u, \gamma)$ is small enough, or for which~\eqref{e:3dSmallness} holds.

\subsection{Relationships between the various dissipation times}\label{s:introDtimeRel}

Since for any~$\alpha, \gamma > 0$, the quantity~$\tau^*_\alpha(u, \gamma)$ is a measure of the rate at which~$u$ mixes, it is natural to study its behavior as $\alpha$ and $\gamma$ vary.
When~$\alpha = 1$, the behavior of~$\tau^*_\alpha(u, \gamma)$ as $\gamma \to 0$ was recently studied in~\cites{CotiZelatiDelgadinoEA18,FengIyer19,Feng19} and quantified in terms of the mixing rate.
We will instead study the behavior of~$\tau^*_\alpha( u, \gamma)$ when~$\gamma$ is fixed and~$\alpha$ varies.
Moreover, since~$\tau^*_1(u, \gamma)$ and $\tau^*_2(u, \gamma)$ are particularly interesting from a physical point of view, we focus our attention on the relationship between these two quantities.
Our first result is an upper bound for $\tau^*_2(u, \gamma)$ in terms of $\tau^*_1(u, \gamma)$.

\begin{lemma}\label{l:tau1tau2}
  There exists an explicit dimensional constant $C$ such that for every divergence-free~$u \in L^\infty( [0, \infty); C^2(\T^d) )$, and every~$\gamma > 0$, we have
  \begin{equation}\label{e:tau2tau1}
    \tau^*_2(u, \gamma)
      \leq C \tau^*_1(u, \gamma) (1 + \norm{u}_{C^2}\, \tau^*_1(u, \gamma) )\,.
  \end{equation}
\end{lemma}

Since velocity fields with small~$\tau^*_1(u, \gamma)$ are known, one use of~Lemma~\ref{l:tau1tau2} is to produce velocity fields for which~$\tau^*_2(u, \gamma)$ and $(1 + \norm{\grad u}_{L^\infty})^{1/2} \tau^*_2(u, \gamma)$ are small.
For instance, if~$u$ is mixing at a sufficiently fast rate, then results of~\cites{Wei18,CotiZelatiDelgadinoEA18,FengIyer19,Feng19} along with Lemma~\ref{l:tau1tau2} can be used to produce velocity fields for which~$\tau^*_2(u, \gamma)$ and $(1 + \norm{\grad u}_{L^\infty})^{1/2} \tau^*_2(u, \gamma)$ are arbitrarily small. Lemma~\ref{l:tau1tau2}, however, cannot be used to produce cellular flows for which $\tau^*_2(u, \gamma)$ is arbitrarily small.
Indeed, with the~$\tau^*_1$ bound in~\cite{IyerXuEA19}, or even the best expected heuristic for cellular flows, the right-hand side of~\eqref{e:tau2tau1} diverges.

\subsection{Plan of the paper}\label{s:introPlan}

In Section~\ref{s:decay} we prove our main result (Theorem~\ref{t:chmix}).
In Section~\ref{s:dtimeMix} we recall the definition of weak and strong mixing and prove Proposition~\ref{p:dtimeMix}.
In Section~\ref{s:relationbetweentau} we prove Lemma~\ref{l:tau1tau2} bounding~$\tau^*_2$ in terms of~$\tau^*_1$.
Finally, for completeness, we conclude with an appendix estimating the dissipation time~$\tau^*_2$ in terms of the mixing rate of the advecting velocity field.

\section{Decay of the advective Cahn--Hilliard equation}\label{s:decay}

This section is devoted to the proof of Theorem~\ref{t:chmix}.
We begin by recalling the well-known existence of global strong solutions to equation~\eqref{e:ach}.
Elliott and Songmu~\cite{ElliottSongmu86} proved well-posedness in the absence of advection.  Since the advection is a first-order linear term, their proof can easily be adapted to our setting.  We state the result here for convenience.

\begin{proposition}\label{p:existence}
  Let  $\gamma > 0$, $u \in L^\infty([0, \infty); W^{1, \infty}(\T^d))$ be divergence-free and $c_0 \in H^2(\T^d)$.
  There exists a unique strong solution to~\eqref{e:ach} in the space
  \begin{equation*}
    c(t,x)\in L^2_{\loc}([0,\infty);H^{4}(\T^{d}))
    \cap L^{\infty}_\loc([0,\infty);H^{2}(\T^{d}))
    \cap H^{1}_\loc([0,\infty); L^{2}(\T^{d}) )\,.
  \end{equation*}

\end{proposition}

For the remainder of this section let $\beta > 1$, $\gamma > 0$, and $\mu > 0$ be as in the statement of Theorem~\ref{t:chmix}.
Without loss of generality we may further assume $\beta \in (1, 2]$.  We also fix a divergence-free velocity field $u \in L^\infty([0, \infty); W^{1, \infty}(\T^d) )$, $c_0 \in H^2(\T^d)$ and let $c$ be the unique strong solution to equation~\eqref{e:ach} with initial data~$c_0$.
The existence of such a solution is guaranteed by Proposition~\ref{p:existence}.

The main idea behind the proof of Theorem~\ref{t:chmix} is to split the analysis into two cases.  First, when the time average of $\norm{\Delta c}_{L^2}$ is large, standard energy estimates will show that the variance of~$c$ decreases exponentially.
Second, when the time average of~$\norm{\Delta c}_{L^2}^2$ is small, we will use the advection term to show that the variance of~$c$ still decreases exponentially, at a comparable rate.

We begin with a lemma handling the first case.

\begin{lemma}\label{l:h1meanlargev2}
For any $t_0 \geq 0$ and $\beta>1$, we have
\begin{equation}\label{e:doubling}
\sup_{0\leqslant\tau\leqslant \gamma \ln \beta}\norm{c({t_{0}+\tau})-\bar{c}}_{L^{2}}^2
\leqslant \beta \norm{c({t_{0}})-\bar{c}}_{L^{2}}^2\,.
\end{equation}
Moreover, if for some $\tau\in (0, \gamma \ln \beta)$ and $\mu>0$ we have
\begin{equation}\label{e:h1meanlarge}
\frac{1}{\tau}\int_{t_0}^{t_0+\tau}\norm{\lap c}_{L^2}^2\,ds \geq \frac{\beta+2\gamma\mu}{\gamma^2}\norm{c\paren{t_0}-\bar c}_{L^2}^2\,,
\end{equation}
then
\begin{equation}\label{e:L2expDecay}
\norm{c(t_0+\tau)-\bar c}_{L^2}\leq e^{-\mu \tau}\norm{c(t_0)-\bar c}_{L^2}\,.
\end{equation}
\end{lemma}

For clarity of presentation, we momentarily postpone the proof of Lemma~\ref{l:h1meanlargev2}.
We will now treat the two- and three-dimensional cases separately.

\subsection{The two-dimensional case}\label{s:2D}

Suppose the time average of~$\norm{\lap c}_{L^2}^2$ is small.
In this case, we will show that if $\tau^*_2(u, \gamma)$ is small enough, then the variance of~$c$ still decreases by a constant fraction after time~$\tau^*_2(u, \gamma)$.

\begin{lemma}\label{l:h1meansmallv2}
For any $t_0\geqslant 0$, there exists a time
\begin{equation*}
  T_0' = T_0'( \norm{c(t_0) - \bar c}_{L^2}, \bar c, \beta,\gamma, \mu)\in (0, \gamma \ln \beta]
\end{equation*}
such that if
\begin{subequations}
\begin{gather}
  \label{e:t2small}
  \tau_2^{*}(u,\gamma)\leqslant T_0'( \norm{c(t_0) - \bar c}_{L^2},\beta, \gamma,\mu,\bar c)\,,
  \\
  \label{e:h2small}
  \frac{1}{\tau_2^*(u,\gamma)}\int_{t_0}^{t_0 + \tau^*_2(u, \gamma)} \norm{\lap c}_{L^2}^2 \, ds
    \leq \frac{\beta + 2 \gamma \mu }{\gamma^2} \norm{c(t_0) - \bar c}_{L^2}^2\,,
\end{gather}
\end{subequations}
then~\eqref{e:L2expDecay} still holds at time~$\tau = \tau^*_2(u, \gamma)$.
Moreover, the time $T_0'$ can be chosen to be decreasing as a function of~$\norm{ c(t_0) - \bar c }_{L^2}$.
\end{lemma}
\begin{remark*}
  The time $T_0'$ can be computed explicitly in terms of $\norm{c(t_0) - \bar c}_{L^2}$, $\beta$, $\gamma$, $\mu$, and $\bar c$, as can be seen from~\eqref{e:T0prime}, below.
\end{remark*}

Momentarily postponing the proof of Lemma~\ref{l:h1meansmallv2}, we prove Theorem~\ref{t:chmix} in 2D.

\begin{proof}[Proof of Theorem~\ref{t:chmix} when $d = 2$]
  Define
  \begin{equation*}
    T_0 = \min \set[\Big]{T_0', \frac{\ln \beta}{2 \mu} }\,,
  \end{equation*}
  where $T_0'$ is the time given by  Lemma~\ref{l:h1meansmallv2} with $t_0 = 0$.
  For conciseness, let $\tau^*_2 = \tau^*_2(u, \gamma)$, and suppose $\tau^*_2 < T_0'$.
  If
  \begin{equation}\label{e:2Dtimeav1}
    \frac{1}{\tau_2^*}\int_0^{\tau^*_2} \norm{\lap c}_{L^2}^2 \, ds
      \geq  \frac{\beta+2\gamma\mu}{\gamma^2}\norm{c\paren{t_0}-\bar c}_{L^2}^2\,,
  \end{equation}
  and since~$T_0' < \gamma \ln \beta$ by choice, Lemma~\ref{l:h1meanlargev2} applies and we must have
  \begin{equation}\label{e:L2tmp1}
    \norm{c(\tau^*_2) - \bar c}_{L^2}
      \leq e^{-\mu \tau^*_2} \norm{c_0 - \bar c}_{L^2}\,.
  \end{equation}
  If on the other hand~\eqref{e:2Dtimeav1} does not hold, then Lemma~\ref{l:h1meansmallv2} applies and~\eqref{e:L2tmp1} still holds.

  Since $T_0'$ is a decreasing function of~$\norm{c - \bar c}_{L^2}$, we may restart the above argument at time $\tau^*_2$.
  Proceeding inductively, we find
  \begin{equation*}
    \norm{c(n\tau_2^*)-\bar c}_{L^2}
      \leq e^{-\mu n\tau_2^*}\norm{c_0-\bar c}_{L^2}\,,
  \end{equation*}
  for all $n \in \N$.

  Now for any time $t \geq 0$, let $n \in \N$ be such that $t \in ( n \tau^*_2, (n+1) \tau^*_2)$.
  Since $t - n \tau^*_2 \leq \tau_2^* \leq \gamma \ln \beta$, Lemma~\ref{l:h1meanlargev2} applies and~\eqref{e:doubling} yields
  \begin{align*}
    \norm{c(t) - \bar c}_{L^2}
      &\leq \sqrt \beta \norm{c(n\tau_2^*) - \bar c}_{L^2}
      \leq \sqrt \beta e^{-\mu n\tau_2^*} \norm{c_0 - \bar c}_{L^2} \\
      &\leq  \sqrt \beta e^{-\mu t +\mu\tau_2^*} \norm{c_0 - \bar c}_{L^2}
      \leq \beta e^{-\mu t}\norm{c_0-\bar c}_{L^2}
      \,.
  \end{align*}
  The last inequality follows from $\tau^*_2 \leq \ln \beta / (2\mu)$.
  This completes the proof.
\end{proof}

\subsection{The three-dimensional case}\label{s:3D}

In this case, in order to prove the analog of Lemma~\ref{l:h1meansmallv2}, we need a stronger assumption on $\tau^*_2(u, \gamma)$.

\begin{lemma}\label{l:h1meansmallv23d}
For any $t_0\geq 0$, there exists a time $T_1' = T_1'(\norm{c(t_0) - \bar c}_{L^2}, \bar c, \beta, \gamma, \mu)$ such that if
\begin{gather}
\label{e:t2small3}
  (1 + \norm{\grad u}_{L^\infty})^{1/2} \tau_2^*(u,\gamma)
    \leq T_1'\,,
    \\
\label{e:h2small3}
\frac{1}{2\tau_2^*(u,\gamma)} \int_{t_0}^{t_0 + 2\tau^*_2(u, \gamma)} \norm{\lap c}_{L^2}^2 \, ds
    \leq \frac{\beta + 2 \gamma \mu }{\gamma^2}\norm{c(t_0)-\bar c}_{L^2}^2\,,
\end{gather}
then
\begin{equation}\label{e:3dexp}
\norm{c(t_0+2\tau_2^*(u,\gamma))-\bar c}_{L^2}\leq e^{-2\mu\tau_2^*(u,\gamma)}\norm{c(t_0)-\bar c}_{L^2}\,.
\end{equation}
Moreover, the time $T_1'$ can be chosen to be decreasing as a function of~$\norm{ c(t_0) - \bar c }_{L^2}$.
\end{lemma}
\begin{remark*}
  The time $T_1'$ can be computed explicitly in terms of $\norm{c(t_0) - \bar c}_{L^2}$, $\beta$, $\gamma$, $\mu$, and $\bar c$, as can be seen from~\eqref{e:T1prime} below.
\end{remark*}

Momentarily postponing the proof of Lemma~\ref{l:h1meansmallv23d}, we prove Theorem~\ref{t:chmix} in 3D.

\begin{proof}[Proof of Theorem~\ref{t:chmix} when $d = 3$]
  Let $T_1'$ be the time given by Lemma~\ref{l:h1meansmallv23d} with $t_0 = 0$, and define
  \begin{equation*}
    T_1 = \min \set[\Big]{ T_1', \frac{\ln \beta}{4 \mu} }\,.
  \end{equation*}
  The remainder of the proof is now identical to the proof when $d = 2$ (page~\pageref{e:2Dtimeav1}) with Lemma~\ref{l:h1meansmallv2} replaced with Lemma~\ref{l:h1meansmallv23d}.
\end{proof}

\subsection{Variance decay in 2D (Lemmas~\ref{l:h1meanlargev2} and~\ref{l:h1meansmallv2})}

It now remains to prove the lemmas.
The variance decay when $\norm{\lap c}_{L^2}$ is large follows directly from the energy inequality in both~2D and~3D.
We prove this first.

\begin{proof}[Proof of Lemma \ref{l:h1meanlargev2}]
  For simplicity and without loss of generality we assume $t_0 = 0$.
Multiplying equation~\eqref{e:ach} by $c-\bar{c}$ and integrating over $\T^d$, we obtain
\begin{align}
  \partial_t \norm{ c-\bar{c}}_{L^{2}}^{2}
    &= 2\ip{\Delta(c^3-c-\gamma\Delta c), c-\bar{c}}\nonumber\\
    &\leq
      -6\norm{ c\nabla c }_{L^{2}}^{2}
      +2\norm{ c-\bar{c}}_{L^{2}}
	\norm{ \lap c }_{L^{2}}
        -2\gamma\norm{ \lap c }_{L^{2}}^{2}\,.
    \label{e:dtL2ineq}
\end{align}
Here the notation $\ip{f, g} = \int_{\T^d} f g \, dx$ denotes the standard $L^2$ inner-product on $\T^d$.  Drop the first term in~\eqref{e:dtL2ineq} and apply Young's inequality to find
\begin{equation}\label{eq:L2 rate bound}
  \partial_t \norm{c-\bar{c}}_{L^{2}}^{2}
    \leq
      -\gamma\norm{\lap c}_{L^{2}}^{2}
      +\frac{1}{\gamma}\norm{c-\bar{c}}_{L^{2}}^{2}\,,
\end{equation}
and hence
\begin{equation}\label{eq:L2rawgrowth}
  \norm{ c(t)-\bar{c}}_{L^{2}}^{2}
    \leq
      \norm{ c_0 -\bar{c}}_{L^{2}}^{2}
      \,e^{t/\gamma}
    \quad \text{for all }t \geq 0\,.
\end{equation}
In particular, if $t \in (0,~  \gamma \ln \beta)$, we see that~\eqref{e:doubling} holds with $t_0 = 0$.

For~\eqref{e:L2expDecay}, note that integration of~\eqref{eq:L2 rate bound} from $0$ to~$\tau$ with~\eqref{e:doubling} and~\eqref{e:h1meanlarge} gives
\begin{equation*}
  \norm{c(\tau) - \bar c}_{L^2}^2
    \leq
      \norm{c_0 - \bar c}_{L^2}^2\paren[\Big]{1 + \frac{\beta \tau}{\gamma} }
      - \gamma \int_0^\tau \norm{\lap c}_{L^2}^2\,ds
    \leq
      \norm{c_0 - \bar c}_{L^2}^2\paren{1 - 2\mu \tau }\,.
\end{equation*}
Since $1 - 2\mu \tau \leq e^{-2\mu \tau}$, this proves~\eqref{e:L2expDecay} as desired.
\end{proof}

We now turn to Lemma~\ref{l:h1meansmallv2}, where the time integral of~$\norm{\lap c}_{L^2}^2$ is assumed small.
In this case, by definition of~$\tau^*_2$, the linear terms halve the variance of~$c$ in time $\tau^*_2$.
If $\tau^*_2$ is small enough, then we show that the nonlinear terms cannot increase the variance too much in this time interval.

\begin{proof}[Proof of Lemma \ref{l:h1meansmallv2}]
  For notational convenience, we use $\mathcal S_{s,t}$ to denote $\mathcal S^{u, 2}_{s,t}$, the solution operator in Definition~\ref{d:dissipationTime} with $\alpha = 2$.
  As before, we also use $\tau_2^*$ to denote $\tau_2^*(u,\gamma)$.
  For simplicity, and without loss of generality, we will again assume $t_0 = 0$.

  By Duhamel's principle, we know
  \begin{equation*}
    c(\tau_2^{*})-\bar c
      = \mathcal S_{0,\tau_2^{*}}(c_0-\bar c)
	+\int_{0}^{\tau_2^{*}}\mathcal S_{s,\tau_2^{*}}(\Delta (c^3(s) - c(s)))\,ds\,.
  \end{equation*}
By definition of~$\tau^*_2 = \tau^*_2(u, \gamma)$, and the fact that $\mathcal S^{u,\alpha}_{s,t}$ is an $L^2$-contraction, we have
  \begin{align}\label{e:l2growth}
    \norm{c(\tau_2^{*})-\bar c}_{L^2}\leq \frac{B}{2}+\int_{0}^{\tau_2^{*}}\norm{\Delta (c^3-c)}_{L^2}\,ds\,,
  \end{align}
  where $B\defeq \norm{c_0-\bar c}_{L^2}$.
  We now estimate the second term on the right of~\eqref{e:l2growth}.
  First note
  \begin{align}\label{e:nonlinearterm}
    \nonumber
    \norm{\Delta (c^3-c)}_{L^2}
    &=\norm{6c\abs{\nabla c}^2+3c^2\Delta c-\Delta c}_{L^2}\\
    \nonumber
    &\leq 6\norm{c-\bar{c}}_{L^\infty}\norm{\nabla c}_{L^4}^2+6\abs{\bar{c}}\norm{\nabla c}_{L^4}^2+3\norm{c-\bar{c}}_{L^\infty}^2\norm{\Delta c}_{L^2}\\
    &\quad+6\abs{\bar{c}}\norm{c-\bar{c}}_{L^\infty}\norm{\Delta c}_{L^2}+(3\bar{c}^2+1)\norm{\Delta c}_{L^2}\,.
  \end{align}
  By the Gagliardo--Nirenberg inequality we know
  \begin{align*}
    \norm{c-\bar c}_{L^\infty}&\leq C\norm{\lap c}_{L^2}^{d/4}\norm{c-\bar c}_{L^2}^{1-d/4}\,,
    \\
    \norm{\nabla c}_{L^4}&\leq C\norm{\lap c}_{L^2}^{(4+d)/8}\norm{c-\bar c}_{L^2}^{(4-d)/8}\,,
  \end{align*}
  for some dimensional constant $C$.
  Here, and subsequently, we assume $C$ is a purely dimensional constant that may increase from line to line.
  Substituting this in~\eqref{e:nonlinearterm} when $d=2$ we find
  \begin{align}\label{e:non1}
    \nonumber
    \norm{\Delta (c^3-c)}_{L^2}
      &\leq C \norm{\lap c}_{L^2}^{2}\norm{c-\bar{c}}_{L^2}
    \\
    \nonumber
      &\qquad
	+ C \abs{\bar{c}}\norm{\lap c}_{L^2}^{3/2}\norm{c-\bar{c}}_{L^2}^{1/2}
	+ (3\bar{c}^2+1)\norm{\lap c}_{L^2}
    \\
      &\leq C
	  (1 + \bar c^2)
	  (1 + \norm{c - \bar c}_{L^2})
	  (\norm{\lap c}_{L^2} + \norm{\lap c}_{L^2}^2)\,.
  \end{align}

  If we choose $T_0'$ small enough to ensure $T_0' < \gamma \ln \beta$, then~\eqref{e:doubling}, \eqref{e:h2small}, \eqref{e:l2growth} and \eqref{e:non1} yield
  \begin{equation*}
    \norm{c(\tau_2^{*})-\bar c}_{L^2}
      \leq \frac{B}{2}
	+ \frac{C_{\beta,\mu} \tau^*_2}{\gamma^2} (1 + \bar c^2) (1 + B^2) B\,.
  \end{equation*}
  Here, $C_{\beta, \mu}$ is a constant that only depends on~$\beta$, $\mu$ that may increase from line to line.
  Now choosing
  \begin{equation}\label{e:T0prime}
    T_0' = \min\set[\Big]{
      \frac{\gamma^2}{4 C_{\beta, \mu} (1 + \bar c^2) (1 + B^2) }\,,\;
      \gamma \ln \beta\,,\;
      \frac{1}{4\mu}
      }
  \end{equation}
  we see that whenever $\tau^*_2 \leq T_0'$ we must have
  \begin{equation*}
    \norm{c(\tau^*_2) - \bar c}_{L^2}
      \leq \frac{3B}{4} \leq (1 - \mu \tau^*_2) B
      \leq e^{-\mu \tau^*_2} \norm{c_0 - \bar c}_{L^2}\,,
  \end{equation*}
  as claimed.
  Clearly the choice of~$T_0'$ above is decreasing in~$B$, finishing the proof.
\end{proof}

\subsection{Variance decay in 3D (Lemma~\ref{l:h1meansmallv23d})}

To prove variance decay in 3D, we first need an $H^1$ bound.
For the remainder of this subsection we assume $d = 3$.

\begin{lemma}\label{l:H1bound}
  Define the \emph{free energy}, $\FE$, by
  \begin{equation*}
    \FE(t)
      \defeq
	\tfrac{1}{4}\int_{\T^3}(c^2-1)^2\,dx
	+\tfrac{1}{2}\gamma \int_{\T^3} \abs{\grad c}^2\,dx\,.
  \end{equation*}
  Then, for any $t_0, \tau \geq 0$ we have
  \begin{equation}\label{e:H1pwbound}
    \norm{\grad c(t_0+\tau)}_{L^2}^2
      \leq \frac{2\FE(t_0)}{\gamma}
	+\frac{\norm{\grad u}_{L^\infty}}{2 \pi^2 \gamma} \, e^{\tau/\gamma}
	\norm{c(t_0)-\bar c}_{L^2}^2\,.
  \end{equation}
\end{lemma}
\begin{proof}
  Without loss of generality assume $t_0 = 0$.
  Multiplying~\eqref{e:ach} by $c^3-c-\gamma \Delta c$ and integrating over~$\T^3$, we have
  \begin{align}\label{e:freeenergy}
    \partial_t \FE
      +\ip{u\cdot \grad c,\; c^3-c-\gamma\Delta c}
      =-\norm{\grad(c^3-c-\gamma \Delta c)}_{L^2}^2\,.
  \end{align}
  Since $u$ is divergence free,
  \begin{equation*}
    \abs{\langle u\cdot \grad c,~c^3-c-\gamma \Delta c\rangle}
      =\abs{\langle u\cdot \grad c, \gamma \Delta c\rangle}
      \leq \gamma \norm{\grad u}_{L^\infty}\norm{\grad c}_{L^2}^2\,.
  \end{equation*}
  Use this in~\eqref{e:freeenergy}, integrate in time then use Poincar\'e's inequality to get
  \begin{align}\label{e:freeenergy1}
    \nonumber
    \int_{0}^{\tau}\norm{\grad(c^3-c-\gamma\Delta c)}^2\,ds
    +\FE(\tau)
      &\leq
	\FE(0)
	+\gamma\norm{\grad u}_{L^\infty}\int_{0}^{\tau}\norm{\grad c}_{L^2}^2\,ds\\
      &\leq
	\FE(0)
	+\frac{\gamma\norm{\grad u}_{L^\infty}}{4 \pi^2}
	  \int_{0}^{\tau}\norm{\Delta c}_{L^2}^2\,ds\,.
  \end{align}

  Time-integrating~\eqref{eq:L2 rate bound} and using~\eqref{eq:L2rawgrowth}, we find
  \begin{align}\label{e:H1bound}
    \gamma \int_{0}^{\tau}\norm{\Delta c}_{L^2}^2\,ds
      \leq
	\frac{1}{\gamma} \int_{0}^{\tau} \norm{c-\bar c}_{L^2}^2\,ds
	+ \norm{c(t_0)-\bar c}_{L^2}^2
      \leq
	e^{\tau/\gamma} \norm{c_0-\bar c}_{L^2}^2\,.
  \end{align}
  Finally, we substitute~\eqref{e:H1bound} in~\eqref{e:freeenergy1} to obtain
  \begin{equation*}
    \tfrac{1}{2}\gamma \norm{\grad c(\tau)}_{L^2}^2
      \leq \FE(\tau)
	\leq \FE(0)
	  +\frac{\norm{\grad u}_{L^\infty}}{4 \pi^2}
	    \, e^{\tau/\gamma}\norm{c_0-\bar c}_{L^2}^2\,,
  \end{equation*}
  which immediately implies~\eqref{e:H1pwbound} as claimed.
\end{proof}

We now prove Lemma~\ref{l:h1meansmallv23d}.
\begin{proof}[Proof of Lemma \ref{l:h1meansmallv23d}]
  As before, we assume without loss of generality that $t_0 = 0$.
  In the 3D case, we will express~$c(2\tau_2^*)$ using Duhamel's principle.
  However, for reasons that will be explained below, we need to use a starting time of~$t_1 \in [0, \tau^*_2]$, which might not be~$0$.  Note that for any $t_1 \in [0, \tau^*_2]$, we have
  \begin{align*}
    c(2\tau_2^{*})-\bar c=\mathcal{S}_{t_1,2\tau_2^{*}}(c(t_1)-\bar c)+\int_{t_1}^{2\tau_2^{*}}\mathcal{S}_{s,2\tau_2^{*}}(\Delta(c^3-c))\,ds\,.
  \end{align*}
  Since $2 \tau^*_2 - t_1 \geq \tau^*_2$, the above implies
  \begin{align}\label{e:t1}
    \norm{c(2\tau_2^{*})-\bar c}_{L^2}
      \leq \tfrac{1}{2} \norm{c(t_1) - \bar c}_{L^2}
	+\int_{t_1}^{2\tau_2^{*}}\norm{\Delta(c^3-c)}\,ds\,.
  \end{align}
  To bound the first term on the right, we note that if $2 T_1' \leq \gamma \ln \beta$, then~\eqref{e:doubling} implies
  \begin{align}\label{e:3dgrowth}
    \norm{c(t_1)-\bar c}_{L^2}
     \leq \sqrt{\beta} B\,,
  \end{align}
  where $B \defeq \norm{c_0 - \bar c}_{L^2}$.

  To bound the second term on the right-hand side, recall the Gagliardo--Nirenberg interpolation inequalities in 3D guarantee
  \begin{align*}
    \norm{c-\bar c}_{L^\infty}
      &\leq C \norm{\grad c}_{L^2}^{1/2} \norm{\Delta c}_{L^2}^{1/2}\,,
    \\
    \norm{\grad c}_{L^4}
      &\leq C \norm{\grad c}_{L^2}^{1/4} \norm{\Delta c}_{L^2}^{3/4}\,.
  \end{align*}
  Expanding $\norm{ \lap( c^3 - c) }_{L^2}$ as in~\eqref{e:nonlinearterm}, and using these inequalities, we see
  \begin{align}\label{e:3dnonlinear}
    \nonumber
    \norm{\Delta(c^3-c)}_{L^2}
      &\leq
	C\norm{\Delta c}_{L^2}^2\norm{\grad c}_{L^2}
	+ C|\bar c| \norm{\Delta c}_{L^2}^{3/2}\norm{\grad c}_{L^2}^{1/2}
    \\
    \nonumber
      &\qquad +(3\bar c^2+1)\norm{\Delta c}_{L^2}\,,
    \\
      &\leq
	C (1+\bar c^2) (1+\norm{\grad c}_{L^2})(\norm{\Delta c}_{L^2}+\norm{\Delta c}_{L^2}^2)\,.
  \end{align}
  The difference from the 2D case is precisely at this step, as the above estimate does not allow us to bound the second term on the right of~\eqref{e:t1} using~\eqref{e:h2small3} and~\eqref{e:doubling} alone.
  Indeed, to bound this term, we now need a time-uniform bound on $\norm{\grad c}_{L^2}$, in combination with~\eqref{e:h2small3} and~\eqref{e:doubling}.
  Unfortunately, the only such bounds we can obtain depend on~$u$, and thus our criterion in 3D involves both~$\norm{\grad u}_{L^\infty}$ and $\tau^*_2$.

  To carry out the details, note first that by Chebyshev's inequality and~\eqref{e:h2small3} we can choose $t_1\in [0, \tau_2^{*}]$ so that
  \begin{equation}\label{e:H2bound}
    \norm{\lap c(t_1)}_{L^2}^2\leq \frac{2\beta+4\gamma \mu}{\gamma^2}B^2\,.
  \end{equation}
  Using the Gagliardo--Nirenberg inequality and~\eqref{e:H2bound} we note that the free energy~$\FE$ at time $t_1$ can be bounded by
  \begin{align*}
    \FE(t_1)
      &\leq \tfrac{1}{4}\norm{c(t_1)}_{L^4}^4
	+ \tfrac{1}{2}\gamma\norm{\grad c(t_1)}_{L^2}^2
	+ \tfrac{1}{4}
    \\
      &\leq
	2\norm{c(t_1)-\bar c}_{L^4}^4+\tfrac{1}{2}\gamma\norm{\grad c}_{L^2}^2+2\bar c^4+\tfrac{1}{4}\\
    &\leq C\norm{c(t_1)-\bar c}_{L^2}^{5/2}\norm{\lap c(t_1)}_{L^2}^{3/2}+\frac{\gamma}{8\pi^2}\norm{\lap c(t_1)}_{L^2}^2+2\bar c^4+\tfrac{1}{4}\\
    &\leq \frac{C \beta^{5/4}(2\beta+4\gamma \mu)^{3/4}}{\gamma^{3/2}}B^4+\frac{\beta+2\gamma \mu}{4\pi^2\gamma}B^2+2\bar c^4+\tfrac{1}{4}
    \\
    &\leq
      \frac{ C_{\beta,\mu}  B^4}{\gamma^2} + 2 (\bar c^4 + 1)\,.
  \end{align*}
  Thus, for any time $t \in [t_1, 2 \tau^*]$, we use Lemma~\ref{l:H1bound} and obtain
  \begin{align}
    \nonumber
    \norm{\grad c(t)}_{L^2}^2
      &\leq \frac{2\FE(t_1)}{\gamma}
	+\frac{\norm{\grad u}_{L^\infty}}{2\pi^2\gamma}\,  e^{(t-t_1)/\gamma}
	\norm{c(t_1)-\bar c}_{L^2}^2
    \\
    \label{e:gradc}
    &\leq
      \frac{C_{\beta,\mu}}{\gamma^{3}}
      (1 + \norm{\grad u}_{L^\infty})
      (1 + \bar c^4)  (B^4 + 1)\,.
  \end{align}

  The use of~\eqref{e:h2small3}, \eqref{e:3dnonlinear} and~\eqref{e:gradc} in~\eqref{e:t1} yields
  \begin{align}\label{e:c2ts1}
    \nonumber
    \norm{c(2\tau_2^{*})-\bar c}_{L^2}
      &\leq \frac{\sqrt{\beta} B}{2}
	+(1+\bar c^2)\paren[\Big]{1+\frac{C_{\beta,\mu}}{\gamma^{3/2}}(1+\norm{\grad u}_{L^\infty})^{1/2}(1+\bar c^2)(1+B^2)}
    \\
    \nonumber
      &\qquad \cdot \paren[\Big]{\frac{C_{\beta,\mu}}{\gamma^{2}}B^2+\frac{C_{\beta,\mu}}{\gamma}B}\tau_2^*
    \\
    &\leq \frac{\sqrt{\beta} B}{2}
	+ \frac{C_{\beta, \mu}\tau^*_2}{\gamma^{7/2}}
	   (1 + \norm{\grad u}_{L^\infty})^{1/2}
	   (1 + \bar c^4) (1 + B^3) B\,.
  \end{align}
  Thus if we choose
  \begin{equation}\label{e:T1prime}
    T_1' \defeq
      \min\set[\Big]{
	\paren[\Big]{\frac{3}{4} - \frac{\sqrt{\beta}}{2}}
	\frac{\gamma^{7/2}}{(1 + B^3) (1 + \bar c^4) C_{\beta,\mu} }\,,\;
	\frac{\gamma \ln \beta}{2}\,,\;
	\frac{1}{8 \mu} }\,,
  \end{equation}
  then our assumption~\eqref{e:t2small3} and the bound~\eqref{e:c2ts1} imply~\eqref{e:3dexp} as claimed.
  Note that, since we have previously assumed $\beta \leq 2$, the choice of $T_1'$ will be strictly positive.
  Finally, the fact that $T_1'$ is decreasing in~$\norm{c_0 - \bar c}_{L^2}$ follows directly from~\eqref{e:T1prime}.
\end{proof}

\section{The dissipation time of mixing flows}\label{s:dtimeMix}

In this section we prove Proposition~\ref{p:dtimeMix}.
Since working on closed Riemannian manifolds introduces almost no added complexity, we will prove Proposition~\ref{p:dtimeMix} in this setting.
Let~$M$ be a $d$-dimensional, smooth, closed Riemannian manifold, with metric normalized so that~$\vol(M) = 1$.
Let $\lap$ denote the Laplace--Beltrami operator on~$M$, and~$u \in L^\infty( [0, \infty); W^{1, \infty}(M) )$ be a divergence-free vector field.
We begin by recalling the definition of \emph{weakly mixing} and \emph{strongly mixing} that we use.

\begin{definition}\label{d:mixing}
  Let $h \colon [0, \infty) \to (0, \infty)$ be a continuous decreasing function that vanishes at $\infty$.
  Given~$\phi_0 \in \dot L^2(M)$, let $\phi$ denote the solution of
  \begin{equation}\label{e:transport}
    \partial_t \phi + u(t,x) \cdot \grad \phi = 0\,,
  \end{equation}
  on~$M$, with initial data~$\phi_0$.
  \begin{enumerate}
    \item
      We say $u$ is \emph{weakly mixing} with rate function~$h$ if for every~$\phi_0, \psi \in \dot H^1(M)$ and every $s, T \geq 0$ we have
      \begin{equation*}
	\paren[\Big]{
	  \frac{1}{T} \int_0^T \abs[\big]{ \ip{\phi(s + t), \psi} }^2 \, dt
	}^{1/2}
	  \leq h(T) \norm{\phi(s)}_{H^1} \norm{\psi}_{H^1}\,.
      \end{equation*}

    \item
      We say $u$ is \emph{strongly mixing} with rate function~$h$ if for every~$\phi_0, \psi \in \dot H^1(M)$ and every $s, t \geq 0$ we have
      \begin{equation*}
	\abs[\big]{ \ip{\phi(s + t), \psi} }
	  \leq h(t) \norm{\phi(s)}_{H^1} \norm{\psi}_{H^1}\,.
      \end{equation*}

  \end{enumerate}
\end{definition}

The use of $H^1$ norms in Definition~\ref{d:mixing} is purely for convenience, and is motivated by~\cites{LinThiffeaultEA11,Thiffeault12,FengIyer19}.
The traditional choice in the dynamical systems literature is to use $C^1$ norms instead.
This difference, however, is not significant as varying the norms used in Definition~\ref{d:mixing} only changes the mixing rate function (see for instance Appendix~A in~\cite{FengIyer19}).

In~\cite{FengIyer19,Feng19} the authors estimated the dissipation time $\tau^*_1(u, \gamma)$ in terms of the weak (or strong) mixing rate function~$h$.
With minor modifications, their work can be modified to give the following estimate for~$\tau^*_2$.

\begin{theorem}\label{t:strong}
   Let $u \in L^\infty( [0, \infty); C^{2}(M) )$ be a divergence-free vector field, and $h \colon [0, \infty) \to (0, \infty)$ be a continuous decreasing function that vanishes at $\infty$.
  \begin{enumerate}\reqnomode
   \item
      There exists constants $C_1, C_2 > 0$ such that if $u$ is weakly mixing with rate function $h$, then for all sufficiently small~$\gamma$ we have
      \begin{equation}\label{e:tau2weak}
	\tau^*_2(u, \gamma)
	  \leq t_* + C_1 \norm{u}_{C^2}\, t_*^2\,.
      \end{equation}
      Here~$t^*$ is the unique solution of
      \begin{equation}\label{e:tau0weak}
	\gamma \norm{u}_{C^2}\, t_*^2 = C_2\, (h(\sfrac{t_*}{\sqrt{2}}))^{8 / (4 + d)}.
      \end{equation}
    \item
      There exists constants $C_1, C_2 > 0$ such that if $u$ is strongly mixing with rate function $h$, then for all sufficiently small~$\gamma$, we have~\eqref{e:tau2weak},
      where~$t_*$ is the unique solution of
      \begin{equation}\label{e:tau0strong}
	\gamma \norm{u}_{C^2}\, t_*^2 = C_2\, h^2(\sfrac{t_*}{2\sqrt{2}})\,.
      \end{equation}
  \end{enumerate}
\end{theorem}
The proof of Theorem~\ref{t:strong} is very similar to that in~\cite[Chapter 4]{Feng19}, and we provide a sketch in Appendix~\ref{s:dtimebound}.
We now prove Proposition~\ref{p:dtimeMix} using Theorem~\ref{t:strong}.

\begin{proof}[Proof of Proposition~\ref{p:dtimeMix}]
  Rescaling time by a factor of~$A$ we immediately see that
  \begin{equation}\label{e:tau2A}
    \tau^*_2( u_A, \gamma ) = \frac{1}{A}\, \tau^*_2\paren[\Big]{ v, \frac{\gamma}{A} }\,.
  \end{equation}
  For the first assertion in Proposition~\ref{p:dtimeMix}, we assume~$v$ is weakly mixing with rate function~$h$.
  Using~\eqref{e:tau2weak} and~\eqref{e:tau2A} we see that
  \begin{equation}\label{e:tau2rescaled}
    \tau^*_2(u_A, \gamma) \leq
      \frac{1}{A} \paren[\big]{
	t_*(A) + C_1 \norm{v}_{C^2}\, t_*^2(A) } \,,
  \end{equation}
  where $t_*(A)$  solves
  \begin{equation}\label{e:tau0Aweak}
	\frac{\gamma}{A}\, \norm{v}_{C^2}\, t_*^2(A) = C_2\, (h(\tfrac{t_*(A)}{\sqrt{2}}))^{8 / (4 + d)}\,.
  \end{equation}
  Clearly this implies $t_*(A) \to \infty$ as $A \to \infty$.
  Since $h$ vanishes at~$\infty$, this in turn implies that $t_*^2(A) / A \to 0$ as $A \to \infty$.
  Consequently, the right hand side of~\eqref{e:tau2rescaled} vanishes as $A \to \infty$, proving the first assertion of Proposition~\ref{p:dtimeMix}.
  \smallskip

  For the second assertion, we assume $v$ is strongly mixing with rate function~$h$ satisfying~\eqref{e:thtVanish}.
  In this case Theorem~\ref{t:strong} and~\eqref{e:tau2A} imply~\eqref{e:tau2rescaled} still holds, provided $t_*(A)$ is defined by
  \begin{equation}\label{e:tau0Astrong}
	\frac{\gamma}{A} \norm{v}_{C^2}\, t_*^2(A) = C_2\, h^2\paren[\Big]{ \frac{t_*(A)}{2\sqrt{2}} }\,.
  \end{equation}
  Note that this still implies~$t_*(A) \to \infty$ as $A \to \infty$.
  Using this along with~\eqref{e:thtVanish} we see that
  \begin{equation*}
    \frac{t_*^2(A)}{A} \leq \frac{\epsilon}{t_*^2(A)}
  \end{equation*}
  for any~$\epsilon>0$, and all sufficiently large~$A$.
  Using this in~\eqref{e:tau2rescaled} yields $A^{1/2} \tau^*_2(u_A, \gamma) \to 0$ as $A \to \infty$, concluding the proof.
\end{proof}

\section{Relationship between \texorpdfstring{$\tau^*_1$}{tau1*} and \texorpdfstring{$\tau^*_2$}{tau2*} (Lemma~\ref{l:tau1tau2})}\label{s:relationbetweentau}

In this section we prove Lemma~\ref{l:tau1tau2} bounding $\tau^*_2(u, \gamma)$ in terms of  $\tau^*_1(u, \gamma)$.
Throughout we fix $u \in L^\infty( [0, \infty); C^2(\T^d) )$, and assume~$\theta$ is a solution of~\eqref{e:ad} with~$\alpha = 2$ and mean-zero initial data~$\theta_0 \in \dot L^2(\T^d)$.
As before, we abbreviate $\tau^*_\alpha( u, \gamma )$ to $\tau^*_\alpha$.

The proof of Lemma~\ref{l:tau1tau2} is similar to that of Theorem~\ref{t:chmix} in 3D.
We divide the analysis into two cases: the first where the time average of $\norm{\lap \theta}_{L^2}^2$ is large (Lemma~\ref{l:comlarge}), and the second where the time average of~$\norm{\lap \theta}_{L^2}^2$ is small (Lemma~\ref{l:comsmall}).  Lemma~\ref{l:tau1tau2} will be proven after these two lemmas.

\begin{lemma}\label{l:comlarge}
  If for some $t_0 \geq 0$, $\lambda, \tau > 0$ we have
  \begin{equation}\label{e:h2big}
    \frac{1}{\tau}\int_{t_0}^{t_0+\tau}\norm{\lap \theta}_{L^2}^2\,ds
      \geq \lambda  \norm{\theta(t_0)}_{L^2}^2\,,
  \end{equation}
  then
  \begin{equation}\label{e:L2eDecay2}
    \norm{\theta (t_0+\tau)}_{L^2}
      \leq e^{-\lambda\gamma \tau}\norm{\theta(t_0)}_{L^2} \,.
  \end{equation}
\end{lemma}
\begin{proof}
  Multiplying~\eqref{e:ad} by~$\theta$ and integrating, we obtain
  \begin{align*}
    \norm{\theta(t_0+ \tau)}_{L^2}^2
      =\norm{\theta(t_0)}_{L^2}^2-2\gamma\int_{t_0}^{t_0+ \tau}\norm{\lap \theta}_{L^2}^2\,ds\,.
  \end{align*}
  Inequalities~\eqref{e:h2big} and $1 - x \leq e^{-x}$ yield~\eqref{e:L2eDecay2} as desired.
\end{proof}

\begin{lemma}\label{l:comsmall}
  There exists an explicit dimensional constant~$C_1$ such that if
  \begin{equation*}
    \lambda \defeq \frac{1}{4 \gamma \tau^*_1 (20C_1 \norm{u}_{C^2}\, \tau^*_1 + 11)}\,,
  \end{equation*}
  and for some $t_0 \geq 0$ we have
  \begin{equation}\label{e:thetaH2small}
    \frac{1}{2\tau^*_1} \int_{t_0}^{t_0+ 2 \tau^*_1}
      \norm{\lap \theta}_{L^2}^2\,ds
      \leq \lambda  \norm{\theta(t_0)}_{L^2}^2\,,
  \end{equation}
  then~\eqref{e:L2eDecay2} still holds at time~$\tau = 2\tau^*_1$.
\end{lemma}
\begin{proof}
  Without loss of generality assume~$t_0 = 0$.
  By Chebyshev's inequality, there exists $t_1 \in [0, \tau^*_1]$ such that
  \begin{equation}\label{e:thetaL21}
    \norm{\lap \theta(t_1)}_{L^2}^2 \leq 2 \lambda \norm{\theta_0}_{L^2}^2\,.
  \end{equation}
  Since
  \begin{equation*}
    \partial_t \theta + u \cdot \grad \theta - \gamma \lap \theta
      = -\gamma \lap^2 \theta - \gamma \lap \theta\,,
  \end{equation*}
  Duhamel's principle implies
  \begin{equation*}
    \theta(2\tau^*_1)
      = \mathcal S^{u, 1}_{t_1, 2\tau^*_1} \theta(t_1)
	-\gamma  \int_{t_1}^{2 \tau^*_1}
	  \mathcal S^{u, 1}_{s, 2\tau^*_1} (\lap^2 \theta(s) + \lap \theta(s) ) \, ds\,,
  \end{equation*}
  where~$\mathcal S$ is the solution operator from Definition~\ref{d:dissipationTime}.
  Since $2 \tau^*_1 - t_1 \geq \tau^*_1$, and $\mathcal S$ is an~$L^2$ contraction, then Poincar\'e's inequality gives
  \begin{align}\label{e:theta2tau1}
    \norm{\theta(2 \tau^*_1)}_{L^2}
      &\leq \frac{\norm{\theta_0}_{L^2}}{2}
	  + 2 \gamma \int_{t_1}^{2 \tau^*_1} \norm{\lap^2 \theta}_{L^2} \, ds\,.
  \end{align}

  To estimate the second term on the right, we multiply~\eqref{e:ad} with $\alpha=2$ by $\lap^2 \theta$ and integrate in space to obtain
  \begin{equation*}
    \tfrac{1}{2} \partial_t \norm{\lap \theta}_{L^2}^2
	+ \gamma \norm{\lap^2 \theta}_{L^2}^2
      \leq C_1 \norm{u}_{C^2}\, \norm{\lap \theta}_{L^2}^2\,,
  \end{equation*}
  for some explicit dimensional constant $C_1$.
  Integration in time together with~\eqref{e:thetaH2small} and~\eqref{e:thetaL21} yields
  \begin{equation*}
    2 \gamma \int_{t_1}^{2\tau^*_1} \norm{\lap^2 \theta}_{L^2}^2 \, ds
      \leq \lambda (4C_1 \norm{u}_{C^2}\, \tau^*_1 + 2) \norm{\theta_0}_{L^2}^2 \,.
  \end{equation*}
  Using this in~\eqref{e:theta2tau1} we have
  \begin{equation*}
    \norm{\theta(2 \tau^*_1)}_{L^2}
      \leq \paren[\Big]{
	  \tfrac{1}{2}
	  + 2\paren[\big]{
	      \gamma \tau^*_1 \lambda (2 + 4C_1 \norm{u}_{C^2}\, \tau^*_1)
	    }^{1/2}
	  } \norm{\theta_0}_{L^2}\,.
  \end{equation*}
  By our choice of~$\lambda$ this implies
  \begin{equation*}
    \norm{\theta(2 \tau^*_1)}_{L^2} \leq (1 - 2 \lambda \gamma \tau^*_1) \norm{\theta_0}_{L^2} \leq e^{-2\lambda \gamma \tau^*_1} \norm{\theta_0}_{L^2}\,,
  \end{equation*}
  finishing the proof.
\end{proof}

The proof of Lemma~\ref{l:tau1tau2} follows quickly from Lemmas~\ref{l:comlarge} and~\ref{l:comsmall}.

\begin{proof}[Proof of Lemma~\ref{l:tau1tau2}]
  Iterating Lemmas~\ref{l:comlarge} and~\ref{l:comsmall} repeatedly we see that for any $t_0 \geq 0$ and $n \in \N$ we have
  \begin{equation*}
    \norm{\theta(t_0 + 2n \tau^*_1 )}_{L^2} \leq e^{-2 n\lambda \gamma \tau^*_1 } \norm{\theta(t_0)}_{L^2}\,.
  \end{equation*}
  Thus we must have~$\tau^*_2 \leq (\ln 2) / (\lambda \gamma)$, from which~\eqref{e:tau2tau1} follows.
\end{proof}

\appendix
\section{Dissipation time bounds of mixing vector fields}\label{s:dtimebound}
In this section, we prove Theorem~\ref{t:strong}.
As in Section~\ref{s:dtimeMix}, we assume here that $M$ is a smooth, closed, Riemannian manifold with volume~$1$, and $\Delta$ is the Laplace--Beltrami operator on~$M$.
We also fix a divergence free vector field~$u \in L^\infty( [0, \infty); C^{2}(M) )$, and let~$\theta$ be the solution to the advection hyper-diffusion equation~\eqref{e:ad} with~$\alpha = 2$ on the manifold~$M$, with mean-zero initial data~$\theta_0 \in \dot L^2(M)$.

The idea behind the proof of Theorem~\ref{t:strong} is to divide the analysis into two cases.
When~$\norm{\lap \theta}_{L^2} / \norm{\theta}_{L^2}$ is large, the energy inequality implies~$\norm{\theta}_{L^2}$ decays rapidly.
On the other hand, when~$\norm{\lap \theta}_{L^2} / \norm{\theta}_{L^2}$ is small, we use the mixing assumption on~$u$ to show that~$\norm{\theta}_{L^2}$ still decays rapidly.
The outline of the proof is the same as that of Theorem~\ref{t:chmix}; however, the proof of the second case is substantially different.
We begin by stating two lemmas handling each of the above cases.

\begin{lemma}\label{l:H2large}
The solution~$\theta$ satisfies the energy inequality
\begin{align}\label{e:energy}
  \partial_t\norm{\theta}_{L^2}^2=-2\gamma \norm{\lap \theta}_{L^2}^2\,.
\end{align}
Consequently, if for some $c_0>0$ we have
\begin{align*}
\norm{\lap \theta(t)}_{L^2}^2 \geq c_0\norm{\theta(t)}_{L^2}^2\,, \quad \text{ for all } 0\leq t \leq t_0\,,
\end{align*}
then
\begin{align}\label{e:H2Large}
\norm{\theta(t)}_{L^2}^2\leq e^{-2\gamma c_0 t}\norm{\theta_0}_{L^2}^2\,, \quad \text{ for all } 0\leq t \leq t_0\,.
\end{align}
\end{lemma}

\begin{lemma}\label{l:H2smallweak}
  Let $0 < \lambda_1 \leq \lambda_2 \leq\cdots$ be the eigenvalues of the Laplacian, where each eigenvalue is repeated according to its multiplicity.
 Suppose~$u$ is weakly mixing with rate function~$h$. There exists positive, finite dimensional constants $\tilde C$, $\tilde c$ such that for all $\gamma$ sufficiently small the following holds:
  If $\lambda_N$ is an eigenvalue of the Laplace--Beltrami operator such that%
  \footnote{
    When $\gamma$ is sufficiently small such a $\lambda_N$ is guaranteed to exist.}
  \begin{align}\label{e:choicelambdaweak}
    h^{-1}\paren[\Big]{ \frac{1}{\tilde c \lambda_N^{(d + 4) / 4} } }
      \leq \frac{1}{\tilde C \lambda_N \sqrt{\gamma} \norm{u}_{C^2}^{1/2} }\,,
  \end{align}
  and if
  \begin{equation}\label{e:H1ByL2SmallCts}
    \norm{\lap \theta_0}_{L^2}^2
	< \lambda_N^2 \norm{\theta_0}_{L^2}^2
  \end{equation}
  holds, then we have
  \begin{equation}\label{e:thetaT0Ctsweak}
    \norm{ \theta(t_0) }_{L^2}^2
      \leq
	\exp\paren[\Big]{- \frac{\gamma \lambda_N^2 t_0}{4} } \norm{\theta_{0}}_{L^2}^2\,,
  \end{equation}
  at a time $t_0$ given by
  \begin{equation}\label{e:t0weak}
    t_0 \defeq h^{-1}\Big(
      \frac{1}{\tilde c \lambda_N^{(d+4)/4}} \Big) \,.
  \end{equation}
\end{lemma}

If instead~$u$ is strongly mixing, then the analog of Lemma~\ref{l:H2smallweak}  is as follows.

\begin{lemma}\label{l:H2small}
 Suppose~$u$ is strongly mixing with rate function~$h$.
  There exists a finite dimensional $\tilde C > 0$ such that for all $\gamma$ sufficiently small the following holds:
  If $\lambda_N$ is an eigenvalue of the Laplace--Beltrami operator such that
  \begin{align}\label{e:choicelambda}
    2 h^{-1} \paren[\Big]{ \frac{1}{2 \lambda_N} }
      \leq \frac{1}{\tilde C \lambda_N \sqrt{\gamma} \norm{u}_{C^2}^{1/2} }\,,
  \end{align}
  and if~\eqref{e:H1ByL2SmallCts} holds,
  then~\eqref{e:thetaT0Ctsweak}  holds at a time $t_0$ given by
  \begin{equation}\label{e:t0strong}
    t_0 \defeq 2 h^{-1} \paren[\Big]{ \frac{1}{2 \lambda_N} } \,.
  \end{equation}
\end{lemma}

Finally, for the proof of Theorem~\ref{t:strong} we need Weyl's Lemma (see for instance~\cite{MinakshisundaramPleijel49}), which describes the asymptotic growth of the eigenvalues of the Laplace--Beltrami operator.

\begin{lemma}[Weyl's Lemma]
Let $0 < \lambda_1 \leq \lambda_2 \leq\cdots$ be the eigenvalues of the Laplacian, where each eigenvalue is repeated according to its multiplicity. We have
 \begin{equation}\label{e:weyl}
      \lambda_j \approx \frac{4\pi\, \Gamma(\frac{d}{2}+1)^{\sfrac{2}{d}}}{\vol(M)^{\sfrac{2}{d}}}\,j^{\sfrac{2}{d}}\,,
    \end{equation}
    asymptotically as $j \to \infty$.
\end{lemma}

Momentarily postponing the proof of Lemmas~\ref{l:H2large}--\ref{l:H2small}, we prove Theorem~\ref{t:strong}.
\begin{proof}[Proof of Theorem~\ref{t:strong}]
  For the first assumption, we assume~$u$ is weakly mixing with rate function~$h$.
  Let~$\tilde c$, $\tilde C$ be the constants from Lemma~\ref{l:H2smallweak}.
  Note that the intermediate value theorem readily implies the existence of a unique~$\lambda_* > 0$ such that
  \begin{equation}\label{e:lambdaStarWeak}
    h^{-1}\paren[\Big]{ \frac{1}{\tilde c \lambda_*^{(d + 4) / 4} } }
      = \frac{1}{\tilde C \lambda_* \sqrt{\gamma} \norm{u}_{C^2}^{1/2} }\,.
  \end{equation}
  Further, it is easy to see that $\lambda_* \to \infty$ as $\gamma \to 0$.
    Thus, for all sufficiently small~$\gamma$, Weyl's lemma implies $\lambda_{j+1}-\lambda_{j}=o(\lambda_j)$ as $j \to \infty$. Hence, for all sufficiently large $\lambda_*$, one can always find~$N$ large enough such that
  \begin{equation}\label{e:lambdaStar}
    \frac{\lambda_*^2}{2} \leq \lambda_N^2 \leq \lambda_*^2 \,.
  \end{equation}

  Now choosing $c_0=\lambda_N^2$ and repeatedly applying Lemmas~\ref{l:H2large} and Lemma~\ref{l:H2smallweak}, we obtain an increasing sequence of times $(t'_k)$, such that $t'_k \to \infty$, $t_{k+1}'-t_k'\leq t_0$, and
  \begin{align*}
    \norm{\theta_s(t_k')}_{L^2}^2 \leq \exp \Big(-\frac{\gamma \lambda_N^2 t_k'}{4}\Big)\norm{\theta_{0}}_{L^2}^2\,.
  \end{align*}
  This immediately implies
  \begin{align}\label{e:tmpCts5}
    \tau_2^*(u,\gamma) \leq \frac{8 \ln 2}{\gamma\lambda_N^2}+t_0\,.
  \end{align}
  Choosing
  \begin{equation*}
    t_* \defeq
      \frac{\sqrt{2}}{\tilde C \lambda_* \sqrt{\gamma} \norm{u}_{C^2}^{1/2} }\,,
  \end{equation*}
  and using~\eqref{e:lambdaStarWeak}, \eqref{e:lambdaStar}, and~\eqref{e:tmpCts5} yields~\eqref{e:tau2weak} as claimed.

  The proof of the second assertion of Theorem~\ref{t:strong} is almost identical to that of the first assertion.
  The only change required is to replace Lemma~\ref{l:H2smallweak} with~\ref{l:H2small}.
\end{proof}

It remains to prove Lemmas~\ref{l:H2large}--\ref{l:H2small}.
\begin{proof}[Proof of Lemma~\ref{l:H2large}]
  Multiplying~\eqref{e:ad} by~$\theta$, integrating over~$M$ and using the fact that~$u$ is divergence free immediately yields~\eqref{e:energy}.
  The second assertion of Lemma~\ref{l:H2large} follows from this and Gronwall's lemma.
\end{proof}

For Lemmas~\ref{l:H2smallweak} and~\ref{l:H2small} we will need a standard result estimating the difference between~$\theta$ and solutions to the inviscid transport equation.

\begin{lemma}\label{l:l2diffcts}
  Let $\phi$ be the solution of~\eqref{e:transport} with initial data $\theta_0$.
  There exists a dimensional constant~$C_d$ such that for all~$t \geq 0$ we have
  \begin{align}\label{e:l2diff}
    \norm{\theta(t)-\phi(t)}_{L^2}^2
      \leq \sqrt{2\gamma t}\,\norm{\theta_0}_{L^2}
	\paren[\Big]{
	  C_d\, \norm{u}_{C^2}\int_0^t\norm{\lap \theta}_{L^2}^2 \,ds
	  +\norm{\lap \theta_0}_{L^2}^2}^{1/2}\,.
  \end{align}
\end{lemma}
\begin{proof}
  Subtracting~\eqref{e:ad} and~\eqref{e:transport} shows
  \begin{align*}
    \partial_t (\theta-\phi)+u\cdot \grad (\theta-\phi)+\gamma \Delta^2\theta=0\,.
  \end{align*}
  Multiplying this by $\theta(t)-\phi(t)$  and integrating over space and time gives
  \begin{align}\label{e:Atmp1}
    \norm{\theta(t)-\phi(t)}_{L^2}^2
      = -2\gamma\int_0^t \int_M (\theta - \phi) \lap^2 \theta\,dx\,ds
      \leq 2\gamma \norm{\theta_0}_{L^2}\int_0^t\norm{\Delta^2 \theta}_{L^2}\,ds\,.
  \end{align}
On the other hand, multiplying~\eqref{e:ad} by~$\Delta^2\theta$ and integrating over~$M$ gives
\begin{align*}
\partial_t \norm{\Delta\theta}_{L^2}^2+2\ip{u\cdot \grad \theta, \Delta^2 \theta}+2\gamma \norm{\Delta^2\theta}_{L^2}^2=0\,.
\end{align*}
Integrating the middle term by parts, using the fact that~$u$ is divergence free, and integrating in time yields
\begin{align*}
2\gamma \int_0^t \norm{\Delta^2\theta}_{L^2}^2\,ds
  \leq C_d\norm{u}_{C^2}\int_0^t\norm{\lap \theta}_{L^2}^2\,ds
    +\norm{\lap \theta_0}_{L^2}^2\,,
\end{align*}
for some dimensional constant $C_d$.
Substituting this in~\eqref{e:Atmp1} and using the Cauchy--Schwartz inequality gives~\eqref{e:l2diff} as claimed.
\end{proof}

We now prove Lemma~\ref{l:H2smallweak}.
\begin{proof}[Proof of Lemma~\ref{l:H2smallweak}]
  We claim that our choice of $\lambda_N$ and $t_0$ will guarantee
  \begin{align}\label{e:lowerbound}
    \int_0^{t_0}\norm{\lap \theta(s)}_{L^2}^2\,ds \geq \frac{\lambda_N^2t_0\norm{\theta_{0}}_{L^2}^2}{8}\,.
  \end{align}
  Once this is established, integrating~\eqref{e:energy} in time immediately yields~\eqref{e:thetaT0Ctsweak}.

  Thus, to prove Lemma~\ref{l:H2smallweak}, we only need to prove~\eqref{e:lowerbound}.
  Suppose, for contradiction, the inequality~\eqref{e:lowerbound} does not hold.
  Letting $P_N \colon \dot L^2(M) \to \dot L^2(M)$ denote the orthogonal projection onto the span of the first~$N$ eigenfunctions of the Laplace--Beltrami operator, we observe
  \begin{align}
    \nonumber
    \MoveEqLeft
    \frac{\lambda_N^2 t_0 \norm{\theta_0}_{L^2}^2}{8}
      > \int_0^{t_0}\norm{\lap \theta(s)}_{L^2}^2\,ds
      \geq  \lambda_N^2\int_{\sfrac{t_0}{2}}^{t_0}\norm{(I-P_N)\theta(s)}_{L^2}^2\,ds
    \\
    \nonumber
      &\geq
	\frac{\lambda_N^2}{2}\int_{\sfrac{t_0}{2}}^{t_0}\norm{(I-P_N)\phi(s)}_{L^2}^2\,ds
	-\lambda_N^2\int_{\sfrac{t_0}{2}}^{t_0}\norm{(I-P_N)\paren[\big]{\theta(s)-\phi(s)}}_{L^2}^2\,ds
    \\
      \label{e:tmpcts2}
      &\geq \frac{\lambda_N^2 t_0}{4}\norm{\theta_{0}}_{L^2}^2-\frac{\lambda_N^2}{2}\int_{\sfrac{t_0}{2}}^{t_0}\norm{P_N \phi(s)}_{L^2}^2\,ds
      -\lambda_N^2\int_{0}^{t_0}\norm{\theta(s)-\phi(s)}_{L^2}^2\,ds\,.
  \end{align}
  We will now bound the last two terms in~\eqref{e:tmpcts2}.

  For the last term in~\eqref{e:tmpcts2}, we use Lemma~\ref{l:l2diffcts} to obtain
  \begin{align}\label{e:tmpCts4}
    \nonumber
    \MoveEqLeft
      \int_{0}^{t_0}\norm{\theta(s)-\phi(s)}_{L^2}^2\,ds
    \\
    \nonumber
      & \leq
	\int_{0}^{t_0}
	  \sqrt{2\gamma s}\,\norm{\theta_0}_{L^2}\paren[\Big]{
	    C_d \norm{u}_{C^2}\int_0^s\norm{\lap \theta(t)}_{L^2}^2\,dt+ \norm{\lap \theta_{0}}_{L^2}^2}^{1/2}\,ds\\
    \nonumber
    &\leq  C \sqrt{\gamma}\, t_0^{3/2} \norm{\theta_{0}}_{L^2}
      \paren[\Big]{
	\norm{u}_{C^2}\int_0^{t_0}\norm{\lap \theta(t)}_{L^2}^2\,dt
	+ \norm{\lap \theta_{0}}_{L^2}^2
      }^{1/2}
    \\
    &\leq C \sqrt{\gamma }\,t_0^{3/2}\lambda_N\norm{\theta_{0}}_{L^2}^2
      \paren[\Big]{\norm{u}_{C^2}\,t_0 +1}^{1/2}\,.
  \end{align}
  For the last inequality above, we used our assumption that the inequality~\eqref{e:lowerbound} does not hold.

  To estimate the second term on the right of~\eqref{e:tmpcts2}, let $e_j$ denote the eigenfunction of the Laplace--Beltrami operator corresponding to the eigenvalue~$\lambda_j$.
  Now
  \begin{align*}
    \int_{\sfrac{t_0}{2}}^{t_0}\norm{P_N\phi(s)}_{L^2}^2\,ds
      \leq \sum_{j = 1}^N \int_0^{t_0} \abs{\ip{\phi(s), e_j}}^2 \, ds
      \leq t_0 h^2(t_0) \norm{\phi_0}_{H^1}^2 \sum_{j= 1}^N \lambda_j \,.
  \end{align*}
  Using Weyl's lemma~\eqref{e:weyl} and the assumption~\eqref{e:H1ByL2SmallCts}, we see
  \begin{equation}\label{e:tmpCts6}
    \int_{\sfrac{t_0}{2}}^{t_0}\norm{P_N\phi(s)}_{L^2}^2\,ds
      \leq C t_0 h^2(t_0) \norm{\phi_0}_{L^2}^2 \lambda_N^{(d + 4) /2}\,,
  \end{equation}
  for some constant~$C = C(M)$.

  We now let $C_1$ be the larger of the constants appearing in~\eqref{e:tmpCts4} and~\eqref{e:tmpCts6}.
  Using these two inequalities in~\eqref{e:tmpcts2} shows
  \begin{equation}\label{e:weakcont}
    \tfrac{1}{8} > \tfrac{1}{4}
      - C_1 \lambda_N \sqrt{\gamma t_0} \paren[\big]{
	  1 + t_0 \norm{u}_{C^2} }^{1/2}
      - C_1 \lambda_N^{(d + 4) / 2} h^2(t_0) \,.
  \end{equation}
  If we choose~$\tilde c \geq \sqrt{16 C_1}$, then by equation~\eqref{e:t0weak} the last term on the right is at most $1/16$.
  Next, when $\gamma$ is sufficiently small we will have $t_0 \norm{u}_{C^2} \geq 1$.
  Thus, if~$\tilde C \geq 16 \sqrt{2} C_1 $ and $\lambda_N$ is the largest eigenvalue for which~\eqref{e:choicelambdaweak} holds, then the second term above is also at most~$1/16$.
  This implies $1/8 > 1/8$, which is the desired contradiction.
\end{proof}

The proof of Lemma~\ref{l:H2small} is very similar to that of Lemma~\ref{l:H2smallweak}.
\begin{proof}[Proof of Lemma~\ref{l:H2small}]
  Follow the proof of Lemma~\ref{l:H2smallweak} until~\eqref{e:tmpCts6}.
  Now, to estimate the second term on the right of~\eqref{e:tmpcts2}, the strongly mixing property of~$u$ gives
  \begin{align}
    \nonumber
    \int_{\sfrac{t_0}{2}}^{t_0}\norm{P_N\phi(s)}_{L^2}^2\,ds
      &\leq \lambda_N\int_{\sfrac{t_0}{2}}^{t_0}\norm{\phi(s)}_{H^{-1}}^2\,ds
      \leq \lambda_N \int_{\sfrac{t_0}{2}}^{t_0}h^2(s) \norm{\theta_{0}}_{H^1}^2\,ds
  \\
    \label{e:tmpCts3weak}
    &\leq  \frac{t_0}{2}\lambda_N h^2(\sfrac{t_0}{2})\,\norm{\lap \theta_0}_{L^2}\norm{\theta_0}_{L^2}
    \leq \frac{t_0}{2}\lambda_N^{2}h^2(\sfrac{t_0}{2})\, \norm{\theta_{0}}_{L^2}^2\,.
  \end{align}
  Above, the last inequality followed from interpolation and the assumption~\eqref{e:H1ByL2SmallCts}.

  Now let $C_1$ be the constant appearing in~\eqref{e:tmpCts4}.
  Using~\eqref{e:tmpCts4} and~\eqref{e:tmpCts3weak} in~\eqref{e:tmpcts2} implies
  \begin{equation*}
    \tfrac{1}{8} > \tfrac{1}{4}
      - C_1 \lambda_N \sqrt{\gamma t_0}
	\paren[\big]{ 1 + t_0 \norm{u}_{C^2} }^{1/2}
      - \tfrac14 {\lambda_N^2 } h^2(\sfrac{t_0}{2}).
  \end{equation*}
  If~$t_0$ is defined by~\eqref{e:t0strong}, then the last term above is at most~$1/16$.
  Moreover, if $\tilde C = 2^{9/2}\, C_1$ and $\lambda_N$ is the largest eigenvalue of the Laplace--Beltrami operator satisfying~\eqref{e:choicelambda}, then the second term above is also at most $1/16$.
  This again forces $1/8 > 1/8$, which is our desired contradiction.
\end{proof}

\bibliographystyle{halpha-abbrv.bst}
\bibliography{refs,preprints}

\newcommand{\etalchar}[1]{$^{#1}$}
\begin{thebibliography}{HMMO95}
\expandafter\ifx\csname url\endcsname\relax
  \def\url#1{\texttt{#1}}\fi
\expandafter\ifx\csname doi\endcsname\relax
  \def\doi#1{\burlalt{doi:#1}{http://dx.doi.org/#1}}\fi
\expandafter\ifx\csname urlprefix\endcsname\relax\def\urlprefix{URL }\fi
\expandafter\ifx\csname href\endcsname\relax
  \def\href#1#2{#2}\fi
\expandafter\ifx\csname burlalt\endcsname\relax
  \def\burlalt#1#2{\href{#2}{#1}}\fi

\bibitem[ACM19]{AlbertiCrippaEA19}
G.~Alberti, G.~Crippa, and A.~L. Mazzucato.
\newblock Exponential self-similar mixing by incompressible flows.
\newblock {\em J. Amer. Math. Soc.}, 32(2):445--490, 2019.
\newblock \doi{10.1090/jams/913}.

\bibitem[BBPS19]{BedrossianBlumenthalEA19}
J.~Bedrossian, A.~Blumenthal, and S.~Punshon-Smith.
\newblock Almost-sure exponential mixing of passive scalars by the stochastic
  {N}avier-{S}tokes equations, 2019,
  \burlalt{1905.03869}{http://arxiv.org/abs/1905.03869}.

\bibitem[Ber01]{Berthier01}
L.~Berthier.
\newblock Phase separation in a homogeneous shear flow: Morphology, growth
  laws, and dynamic scaling.
\newblock {\em Physical Review E}, 63(5):051503, 2001.

\bibitem[BH17]{BedrossianHe17}
J.~Bedrossian and S.~He.
\newblock Suppression of blow-up in {P}atlak-{K}eller-{S}egel via shear flows.
\newblock {\em SIAM J. Math. Anal.}, 49(6):4722--4766, 2017.
\newblock \doi{10.1137/16M1093380}.

\bibitem[BKNR10]{BerestyckiKiselevEA10}
H.~Berestycki, A.~Kiselev, A.~Novikov, and L.~Ryzhik.
\newblock The explosion problem in a flow.
\newblock {\em J. Anal. Math.}, 110:31--65, 2010.
\newblock \doi{10.1007/s11854-010-0002-7}.

\bibitem[Bra03]{Bray03}
A.~Bray.
\newblock Coarsening dynamics of phase-separating systems.
\newblock {\em Philosophical Transactions of the Royal Society of London.
  Series A: Mathematical, Physical and Engineering Sciences},
  361(1805):781--792, 2003.

\bibitem[Cah61]{Cahn61}
J.~W. Cahn.
\newblock On spinodal decomposition.
\newblock {\em Acta Metallurgica}, 9(9):795--801, 1961.
\newblock \doi{10.1016/0001-6160(61)90182-1}.

\bibitem[CH58]{CahnHilliard58}
J.~W. Cahn and J.~E. Hilliard.
\newblock Free energy of a nonuniform system. {I}. {I}nterfacial free energy.
\newblock {\em The Journal of Chemical Physics}, 28(2):258--267, 1958.
\newblock \doi{10.1063/1.1744102}.

\bibitem[CKRZ08]{ConstantinKiselevEA08}
P.~Constantin, A.~Kiselev, L.~Ryzhik, and A.~Zlato{\v{s}}.
\newblock Diffusion and mixing in fluid flow.
\newblock {\em Ann. of Math. (2)}, 168(2):643--674, 2008.
\newblock \doi{10.4007/annals.2008.168.643}.

\bibitem[CPB88]{ChanPerrotEA88}
C.~K. Chan, F.~Perrot, and D.~Beysens.
\newblock Effects of hydrodynamics on growth: {S}pinodal decomposition under
  uniform shear flow.
\newblock {\em Phys. Rev. Lett.}, 61:412--415, Jul 1988.
\newblock \doi{10.1103/PhysRevLett.61.412}.

\bibitem[CZDE18]{CotiZelatiDelgadinoEA18}
M.~Coti~Zelati, M.~G. Delgadino, and T.~M. Elgindi.
\newblock On the relation between enhanced dissipation time-scales and mixing
  rates.
\newblock {\em ArXiv e-prints}, June 2018,
  \burlalt{1806.03258}{http://arxiv.org/abs/1806.03258}.

\bibitem[DEIJ19]{DrivasElgindiEA19}
T.~D. Drivas, T.~M. Elgindi, G.~Iyer, and I.-J. Jeong.
\newblock Anomalous dissipation in passive scalar transport.
\newblock {\em arXiv e-prints}, Nov. 2019,
  \burlalt{1911.03271}{http://arxiv.org/abs/1911.03271}.

\bibitem[Ell89]{Elliott89}
C.~M. Elliott.
\newblock The {C}ahn-{H}illiard model for the kinetics of phase separation.
\newblock In {\em Mathematical models for phase change problems (\'{O}bidos,
  1988)}, volume~88 of {\em Internat. Ser. Numer. Math.}, pages 35--73.
  Birkh\"{a}user, Basel, 1989.

\bibitem[ES86]{ElliottSongmu86}
C.~M. Elliott and Z.~Songmu.
\newblock On the {C}ahn-{H}illiard equation.
\newblock {\em Arch. Rational Mech. Anal.}, 96(4):339--357, 1986.
\newblock \doi{10.1007/BF00251803}.

\bibitem[EZ19]{ElgindiZlatos19}
T.~M. Elgindi and A.~Zlato\v{s}.
\newblock Universal mixers in all dimensions.
\newblock {\em Adv. Math.}, 356:106807, 33, 2019.
\newblock \doi{10.1016/j.aim.2019.106807}.

\bibitem[Fen19]{Feng19}
Y.~Feng.
\newblock {\em Dissipation enhancement by mixing}.
\newblock Carnegie Mellon University, 2019.
\newblock Ph.D. Thesis.

\bibitem[FI19]{FengIyer19}
Y.~Feng and G.~Iyer.
\newblock Dissipation enhancement by mixing.
\newblock {\em Nonlinearity}, 32(5):1810--1851, 2019.
\newblock \doi{10.1088/1361-6544/ab0e56}.

\bibitem[FKR06]{FannjiangKiselevEA06}
A.~Fannjiang, A.~Kiselev, and L.~Ryzhik.
\newblock Quenching of reaction by cellular flows.
\newblock {\em Geom. Funct. Anal.}, 16(1):40--69, 2006.
\newblock \doi{10.1007/s00039-006-0554-y}.

\bibitem[FW03]{FannjiangWoowski03}
A.~Fannjiang and L.~{Wo\l owski}.
\newblock Noise induced dissipation in {L}ebesgue-measure preserving maps on
  {$d$}-dimensional torus.
\newblock {\em J. Statist. Phys.}, 113(1-2):335--378, 2003.
\newblock \doi{10.1023/A:1025787124437}.

\bibitem[HL09]{HouLei09}
T.~Y. Hou and Z.~Lei.
\newblock On the stabilizing effect of convection in three-dimensional
  incompressible flows.
\newblock {\em Comm. Pure Appl. Math.}, 62(4):501--564, 2009.
\newblock \doi{10.1002/cpa.20254}.

\bibitem[HMMO95]{HashimotoMatsuzakaEA95}
T.~Hashimoto, K.~Matsuzaka, E.~Moses, and A.~Onuki.
\newblock String phase in phase-separating fluids under shear flow.
\newblock {\em Physical review letters}, 74(1):126, 1995.

\bibitem[IXZ19]{IyerXuEA19}
G.~Iyer, X.~Xu, and A.~Zlato\v{s}.
\newblock Convection-induced singularity suppression in the {K}eller-{S}egel
  and other non-linear {PDEs}.
\newblock {\em arXiv e-prints}, Aug 2019,
  \burlalt{1908.01941}{http://arxiv.org/abs/1908.01941}.

\bibitem[KSZ08]{KiselevShterenbergEA08}
A.~Kiselev, R.~Shterenberg, and A.~Zlato\v{s}.
\newblock Relaxation enhancement by time-periodic flows.
\newblock {\em Indiana Univ. Math. J.}, 57(5):2137--2152, 2008.
\newblock \doi{10.1512/iumj.2008.57.3349}.

\bibitem[KX16]{KiselevXu16}
A.~Kiselev and X.~Xu.
\newblock Suppression of chemotactic explosion by mixing.
\newblock {\em Arch. Ration. Mech. Anal.}, 222(2):1077--1112, 2016.
\newblock \doi{10.1007/s00205-016-1017-8}.

\bibitem[LDE{\etalchar{+}}13]{LiuDedeEA13}
J.~Liu, L.~Ded\`e, J.~A. Evans, M.~J. Borden, and T.~J. Hughes.
\newblock Isogeometric analysis of the advective {C}ahn--{H}illiard equation:
  {S}pinodal decomposition under shear flow.
\newblock {\em Journal of Computational Physics}, 242:321 -- 350, 2013.
\newblock \doi{10.1016/j.jcp.2013.02.008}.

\bibitem[LLG95]{LaugerLaubnerEA95}
J.~L\"auger, C.~Laubner, and W.~Gronski.
\newblock Correlation between shear viscosity and anisotropic domain growth
  during spinodal decomposition under shear flow.
\newblock {\em Phys. Rev. Lett.}, 75:3576--3579, Nov 1995.
\newblock \doi{10.1103/PhysRevLett.75.3576}.

\bibitem[LTD11]{LinThiffeaultEA11}
Z.~Lin, J.-L. Thiffeault, and C.~R. Doering.
\newblock Optimal stirring strategies for passive scalar mixing.
\newblock {\em J. Fluid Mech.}, 675:465--476, 2011.
\newblock \doi{10.1017/S0022112011000292}.

\bibitem[MD18]{MilesDoering18}
C.~J. Miles and C.~R. Doering.
\newblock Diffusion-limited mixing by incompressible flows.
\newblock {\em Nonlinearity}, 31(5):2346, 2018.
\newblock \doi{10.1088/1361-6544/aab1c8}.

\bibitem[MP49]{MinakshisundaramPleijel49}
S.~Minakshisundaram and {\AA}.~Pleijel.
\newblock Some properties of the eigenfunctions of the {L}aplace-operator on
  {R}iemannian manifolds.
\newblock {\em Canadian J. Math.}, 1:242--256, 1949.
\newblock \doi{10.4153/CJM-1949-021-5}.

\bibitem[ONT07a]{ONaraighThiffeault07}
L.~\'O~N\'araigh and J.-L. Thiffeault.
\newblock Bubbles and filaments: Stirring a {C}ahn--{H}illiard fluid.
\newblock {\em Phys. Rev. E}, 75:016216, Jan 2007.
\newblock \doi{10.1103/PhysRevE.75.016216}.

\bibitem[ONT07b]{ONaraighThiffeault07a}
L.~\'O~N\'araigh and J.-L. Thiffeault.
\newblock Dynamical effects and phase separation in cooled binary fluid films.
\newblock {\em Phys. Rev. E}, 76:035303, Sept. 2007.
\newblock \doi{10.1103/PhysRevE.76.035303}.

\bibitem[ONT08]{ONaraighThiffeault08}
L.~\'O~N\'araigh and J.-L. Thiffeault.
\newblock Bounds on the mixing enhancement for a stirred binary fluid.
\newblock {\em Physica D}, 237(21):2673--2684, Nov. 2008.
\newblock \doi{10.1016/j.physd.2008.04.012}.

\bibitem[Peg89]{Pego89}
R.~L. Pego.
\newblock Front migration in the nonlinear {C}ahn-{H}illiard equation.
\newblock {\em Proc. Roy. Soc. London Ser. A}, 422(1863):261--278, 1989.

\bibitem[Pie94]{Pierrehumbert94}
R.~Pierrehumbert.
\newblock Tracer microstructure in the large-eddy dominated regime.
\newblock {\em Chaos, Solitons \& Fractals}, 4(6):1091--1110, 1994.

\bibitem[Poo96]{Poon96}
C.-C. Poon.
\newblock Unique continuation for parabolic equations.
\newblock {\em Comm. Partial Differential Equations}, 21(3-4):521--539, 1996.
\newblock \doi{10.1080/03605309608821195}.

\bibitem[SC00]{ShouChakrabarti00}
Z.~Shou and A.~Chakrabarti.
\newblock Ordering of viscous liquid mixtures under a steady shear flow.
\newblock {\em Physical Review E}, 61(3):R2200, 2000.

\bibitem[Thi12]{Thiffeault12}
J.-L. Thiffeault.
\newblock Using multiscale norms to quantify mixing and transport.
\newblock {\em Nonlinearity}, 25(2):R1--R44, 2012.
\newblock \doi{10.1088/0951-7715/25/2/R1}.

\bibitem[Wei18]{Wei18}
D.~Wei.
\newblock Diffusion and mixing in fluid flow via the resolvent estimate.
\newblock {\em arXiv e-prints}, Nov 2018,
  \burlalt{1811.11904}{http://arxiv.org/abs/1811.11904}.

\bibitem[YZ17]{YaoZlatos17}
Y.~Yao and A.~Zlato\v{s}.
\newblock Mixing and un-mixing by incompressible flows.
\newblock {\em J. Eur. Math. Soc. (JEMS)}, 19(7):1911--1948, 2017.
\newblock \doi{10.4171/JEMS/709}.

\bibitem[Zla10]{Zlatos10}
A.~Zlato\v{s}.
\newblock Diffusion in fluid flow: dissipation enhancement by flows in 2{D}.
\newblock {\em Comm. Partial Differential Equations}, 35(3):496--534, 2010.
\newblock \doi{10.1080/03605300903362546}.

\end{thebibliography}
\end{document}